\documentclass[a4paper,11pt]{article}       
\usepackage[utf8]{inputenc}     
\usepackage{enumerate}          
\usepackage{amssymb,amsmath,amsfonts,amsxtra,amsthm}            
\usepackage{pdfsync}
\usepackage{xcolor}
\usepackage{url}
\usepackage{mathrsfs} 
\usepackage{graphicx}
\usepackage{hyperref}
\usepackage{anysize}
\usepackage{caption}
\usepackage{subcaption}

\begin{document}
\theoremstyle{plain}
\newtheorem{definition}{Definition}
\newtheorem{lemma}[definition]{Lemma}
\newtheorem{theorem}[definition]{Theorem}
\newtheorem{thm}[definition]{Theorem}
\newtheorem{metatheorem}[equation]{(Meta)-theorem}
\newtheorem{proposition}[definition]{Proposition}
\newtheorem{corollary}[definition]{Corollary}
\newtheorem{conjecture}[equation]{Conjecture}
\newtheorem{problem}{Problem}
\newtheorem{assumption}[equation]{Assumption}
\newtheorem{question}[equation]{Question}
\newtheorem{remark}[definition]{Remark}
\newtheorem{exa}[definition]{Example}
\newtheorem{notation}[equation]{Notation}
\newtheorem{example}[equation]{Important example}
\errorcontextlines=0

\newcommand{\R}{\mathbb{R}}
\newcommand{\Op}{{\rm Op}}
\newcommand{\N}{\mathbb{N}}
\renewcommand{\S}{\mathbb{S}}
\newcommand{\Z}{\mathbb{Z}}
\newcommand{\Id}{\text{Id}}
\newcommand{\T}{\mathbb{T}}
\renewcommand{\sp}{\text{sp}}
\newcommand{\sR}{\text{sR}}
\newcommand{\e}{\varepsilon}
\newcommand{\End}{\text{End}}
\newcommand{\phg}{{\rm phg}}
\newcommand{\supp}{\text{supp}}
\renewcommand{\hom}{\text{hom}}
\renewcommand{\leq}{\leqslant}
\renewcommand{\geq}{\geqslant}
\newcommand{\Hess}{{\rm Hess\;} }
\newcommand{\todo}[1]{$\clubsuit$ {\tt #1}}

\theoremstyle{remark}
\newtheorem{note}[equation]{Note}

\title{Propagation of singularities for subelliptic wave equations}
\author{Cyril Letrouit\footnote{DMA, \'Ecole normale sup\'erieure, CNRS, PSL Research University, 75005 Paris (\texttt{cyril.letrouit@ens.psl.eu})}}
\date{\today}
\maketitle
\abstract{Hörmander's propagation of singularities theorem does not fully describe the propagation of singularities in subelliptic wave equations, due to the existence of doubly characteristic points. In the present work, building upon a visionary conference paper by R. Melrose \cite{Mel86}, we prove that singularities of subelliptic wave equations only propagate along null-bicharacteristics and abnormal extremals, which are well-known curves in optimal control theory. As a consequence, we characterize the singular support of subelliptic wave kernels outside the diagonal. These results show that abnormal extremals play an important role in the classical-quantum correspondence between sub-Riemannian geometry and sub-Laplacians. }

\setcounter{tocdepth}{1}
\tableofcontents

\section{Introduction}\label{s:intro}

\subsection{Motivations} \label{s:motivations}
In microlocal analysis, the celebrated propagation of singularities theorem describes an invariance property for the singularities of the (distributional) solutions of a general class of PDEs. The singularities are encapsulated in the $C^\infty$-wave-front set  (see \eqref{e:WF}),  whose projection is the singular support of the solution. Precisely, if $u$ is a distributional solution to a partial (or pseudo-) differential equation $Pu=f$ and $p$ is the principal symbol of $P$, assumed to be real and homogeneous, this theorem asserts that  $WF(u)\setminus WF(f)\subset p^{-1}(0)$, and $WF(u)\setminus WF(f)$ is invariant under the Hamiltonian flow induced by $p$.

\smallskip

This result was first proved in \cite[Theorem 6.1.1]{DH} and \cite[Proposition 3.5.1]{Ens}. Its second part about invariance of the wave-front set is however not totally satisfactory: it does not provide any information in the case where the characteristics of $P$ are not simple, i.e., at points of the cotangent bundle outside the null section for which $p=0$ and $dp=0$ (since at these points, the Hamiltonian vector field of $p$ vanishes). In a very short and impressive conference paper \cite{Mel86}, Melrose sketched the proof of an analogous propagation of singularities result for the wave operator $P=D_t^2-A$ when $A$ is a self-adjoint non-negative real second-order differential operator which is only subelliptic. Such operators $P$ are typical examples for which there exist double characteristic points.

\smallskip

Despite the potential scope of this result, we did not find in the literature any other paper mentioning it, although several papers make reference to other results contained in \cite{Mel86}. The proof provided in \cite{Mel86} is very sketchy, and we thought it would deserve to be written in full details. This is what we do in the first part of the present work (Sections \ref{s:thecones}, \ref{s:poscommutator} and \ref{s:proof}). 

\smallskip      

Then, pushing further the computations of \cite{Mel86} in the case where $A$ is a sub-Laplacian (see Definition \ref{e:sL}), we prove in Sections \ref{s:xyproof}, \ref{s:th1} and Appendix \ref{s:what} that singularities of subelliptic wave equations driven by sub-Laplacians only propagate along null-bicharacteristics and abnormal extremals, which are curves arising as optimal trajectories in control theory.

\smallskip

In summary, the different objects that will be involved in the statements are the following:
\begin{itemize}
\item The \emph{bicharacteristic flow} is the flow induced by the Hamiltonian vector field of $p$. The Hamiltonian curves of $p$ are called \emph{bicharacteristics}.
\item The \emph{null-bicharacteristics} are the bicharacteristics that are included in $p^{-1}(0)$.
\item The \emph{null-rays} are a larger set of trajectories (in the cotangent bundle) along which the singularities propagate for subelliptic wave equations. Null-bicharacteristics are special instances of null-rays. When $P=D_t^2-A$ with $A$ elliptic, all null-rays are null-bicharacteristics. However, this is not always the case for $A$ subelliptic. These null-rays are defined in Definition \ref{d:defray}.
\item The \emph{abnormal extremals}, defined when $A$ is a sub-Laplacian, are the only null-rays that are not null-bicharacteristics. See Section \ref{s:sR} for a precise definition.
\end{itemize}

\subsection{Statements}

We now state our main results. For the sake of coherence, we borrow nearly all notations to \cite{Mel86}. $A$ is a self-adjoint non-negative real second-order differential operator on a smooth compact manifold $X$ without boundary:
\begin{equation} \label{e:selfadj}
\forall u\in C^\infty(X), \qquad (Au,u)=(u,Au)\geq 0
\end{equation}
with
\begin{equation} \label{e:scalar}
(u,v)=\int_X u(x)\overline{v(x)} d\nu,
\end{equation}
where $\nu$ is some positive $C^\infty$ density. The associated norm is denoted by $\|\cdot\|$.

\smallskip

We also assume that $A$ is subelliptic, in the following sense:
 there exist a (Riemannian) Laplacian $\Delta$ on $X$ and $c,s>0$ such that 
\begin{equation} \label{e:subelliptic}
\forall u\in C^\infty(X), \qquad \|(-\Delta)^{s/2}u\|^2\leq c((Au,u)+\|u\|^2).
\end{equation}
Finally, we assume that $A$ has vanishing subprincipal symbol.

\smallskip

The assumption \eqref{e:selfadj} implies that $A$ has a self-adjoint extension with the domain
\begin{equation*}
\mathscr{D}(A)=\{u\in\mathcal{D}'(X); \ Au\in L^2(X)\}.
\end{equation*}
 By the spectral theorem, for any $t\in\R$, the self-adjoint operator
\begin{equation*}
G(t)=A^{-1/2}\sin(tA^{1/2})
\end{equation*}
is a well-defined  operator bounded on $L^2(X)$, in fact it maps $L^2(X)$ into $\mathscr{D}(A^{1/2})$. Together with the self-adjoint operator $G'(t)=\cos(tA^{1/2})$, this allows to solve the Cauchy problem for the wave operator $P=D_t^2-A$ where $D_t=\frac{1}{i}\partial_t$:
\begin{equation} \label{e:subwave}
\left\lbrace \begin{array}{ll}
D_t^2u-Au=0 &\qquad \text{in } \R\times X, \\
(u_{|t=0},\partial_tu_{|t=0})=(u_0,u_1) & \qquad \text{in } X
\end{array}\right.
\end{equation}
by
\begin{equation*}
u(t,x)=G'(t)u_0+G(t)u_1.
\end{equation*}
For $(u_0,u_1)\in\mathscr{D}(A^{1/2})\times L^2(X)$, we have $u\in C^0(\R;\mathscr{D}(A^{1/2}))\cap C^1(\R;L^2(X))$.

\smallskip

For $f\in\mathcal{D}'(Y)$ a distribution on a manifold $Y$ (equal to $X$, $\R\times X$ or $\R\times X\times X$ in the sequel), we denote by $WF(f)$ the usual Hörmander wave-front set (see \cite{FIOs}):
\begin{equation} \label{e:WF}
WF(f)=\{(y,\eta)\in T^*Y\setminus 0, \ \exists A\in \Psi_{\phg}^0(Y), \text{$A$ elliptic at $(y,\eta)$ and $Af\in C^\infty(Y)$}\}.
\end{equation}
Here an in all the sequel $T^*Y\setminus 0$ denotes the cotangent bundle from which the null-section has been removed. The set $\Psi_{\phg}^0(Y)$ is the set of $0$-th order polyhomogeneous pseudodifferential operators (see Appendix \ref{a:pseudo}). We also recall that the projection through the canonical projection onto $Y$ of $WF(f)$ is the singular support of $f$.

\smallskip

Section 6 in \cite{Mel86} is a sketch of proof for the following statement which characterizes the propagation of singularities in \eqref{e:subwave}:
\begin{thm} \label{t:singprop} 
Let $A$ be a self-adjoint non-negative real second-order differential operator which is subelliptic \eqref{e:subelliptic} and has vanishing subprincipal symbol.
Let $t\mapsto u(t)$ be a solution of \eqref{e:subwave}. For any $t>0$, if $(x,\xi)\in WF(u(0))$ then there exists $(y,\eta)\in WF(u(-t))\cup WF(\partial_tu(-t))$ such that $(y,\eta)$ and $(x,\xi)$ can be joined by a null-ray of length $t$.
\end{thm}
\noindent As mentioned above, null-rays will be defined in Definition \ref{d:defray}.

An important class of examples of operators $A$ satisfying \eqref{e:selfadj}, \eqref{e:scalar}, \eqref{e:subelliptic} and with vanishing sub-principal symbol is given by sub-Laplacians (or Hörmander's sums of squares, see \cite{RS76} or \cite{LL21}):
\begin{definition} \label{e:sL}
A sub-Laplacian is an operator of the form 
\begin{equation}\label{e:AsL}
A=\sum_{i=1}^K Y_i^*Y_i
\end{equation}
for some smooth vector fields $Y_i$ on $X$ satisfying Hörmander's condition: the Lie algebra generated by $Y_1,\ldots,Y_K$ is equal to the whole tangent bundle $TX$. In \eqref{e:AsL}, $Y_i^*$ denotes the adjoint of $Y_i$ for the scalar product \eqref{e:scalar}. The vector fields $Y_i$ are not assumed to be independent. 
\end{definition}

If $A$ is a sub-Laplacian, the null-rays of Theorem \ref{t:singprop} have a particularly simple geometric interpretation. In this case, denoting by $a$ the principal symbol of $A$ and by $H_a$ the associated Hamiltonian vector field, Theorem \ref{t:singprop} can be reformulated as follows (the notions of abnormal extremal lift, singular curve, and associated length are introduced in Section \ref{s:sR}):
\begin{corollary}\label{c:singpropag}
Assume that $A$ is a sub-Laplacian as in Definition \ref{e:sL}. Let $t\mapsto u(t)$ be a solution of \eqref{e:subwave}. For any $t>0$, if $(x,\xi)\in WF(u(0))$ then there exists $(y,\eta)\in WF(u(-t))\cup WF(\partial_tu(-t))$ such that $(y,\eta)$ and $(x,\xi)$ can be joined 
\begin{itemize}
\item either by an Hamiltonian curve: $(x,\xi)=e^{tH_a}(y,\eta)$;
\item or by an abnormal extremal lift of a singular curve of length $\geq t$.
\end{itemize}
\end{corollary}
A weakness of Corollary \ref{c:singpropag} is that it only describes ``from which region of phase space a singularity possibly comes'', but does not assert that singularities effectively propagate along abnormal extremal lifts of singular curves. In particular, the inequality $\geq t$ in the last part of the statement means that singularities could possibly propagate at any speed between $0$ and $1$ along singular curves, but does not prove that it is effectively the case. In a joint work with Yves Colin de Verdi\`ere \cite{CdVL21}, we give explicit examples of initial data of a subelliptic wave equation whose singularities effectively propagate at any speed between $0$ and $1$ along a singular curve. The sub-Laplacian which we use in \cite{CdVL21} is called the Martinet sub-Laplacian (see Example \ref{e:Martinet}). The propagation at speeds between $0$ and $1$ along singular curves is in strong contrast with the propagation ``at speed $1$'' along the integral curves of $H_a$ (as in Hörmander's theorem recalled above).

\smallskip

Theorem 1.8 in \cite{Mel86} (given without proof in \cite{Mel86}, since the sketch of proof in \cite[Section 6]{Mel86} in fact corresponds to Theorem \ref{t:singprop}), which we provide here only in the context of sub-Laplacians, concerns the Schwartz kernel $K_G$ of $G$, i.e., the distribution $K_G\in \mathcal{D}'(\R\times X\times X)$ defined by
\begin{equation}\label{e:Schwartzkernel}
\forall u\in C^\infty(X), \qquad G(t)u(x)=\int_X K_G(t,x,y)u(y)dy.
\end{equation}

\begin{thm} \label{t:inclusion}
Assume that $A$ is a sub-Laplacian as in Definition \ref{e:sL}. Then $WF(K_G)$ is contained in the set of $(t,x,y,\tau,\xi,-\eta)\in T^*(\R\times X\times X)\setminus 0$ such that the following two conditions are satisfied:
\begin{enumerate}[(i)]
\item $\tau^2=a(x,\xi)=a(y,\eta)$;
\item $(y,\eta)$ and $(x,\xi)$ can be joined 
\begin{itemize}
\item either by an Hamiltonian curve: $(x,\xi)=e^{tH_a}(y,\eta)$;
\item or by an abnormal extremal lift of a singular curve of length $\geq t$.
\end{itemize}
\end{enumerate}
\end{thm}
Theorem \ref{t:inclusion} will be deduced from Corollary \ref{c:singpropag} by considering $K_G$ itself as a solution of a subelliptic wave equation. 
The projections on $M$ of integral curves of $H_a$ are called \emph{normal geodesics}. By an adequate projection, we obtain the following corollary in the spirit of the Duistermaat-Guillemin trace formula \cite{DG75}:
\begin{corollary} \label{c:xy}
We fix $x,y\in X$ with $x\neq y$. We denote by $\mathscr{L}$ the set of lengths of normal geodesics from $x$ to $y$ and by $T_{s}$ the minimal length of a singular curve joining $x$ to $y$. Then $\mathscr{G}:t\mapsto K_G(t,x,y)$ is well-defined as a distribution on $(-T_{s},T_{s})$, and
\begin{equation*}
{\rm Sing Supp}(\mathscr{G})\subset \mathscr{L}\cup -\mathscr{L}.
\end{equation*}
\end{corollary}
Note that this corollary does not say anything about times $|t|\geq T_{s}$.

\subsection{Comments, related literature and open questions}

\paragraph{Null-rays.} The null-rays which appear in the statement of Theorem \ref{t:singprop} are generalizations of the usual null-bicharacteristics, which are the integral curves of the Hamiltonian vector field $H_p$ of the principal symbol $p=\tau^2-a$ of $P$ contained in the characteristic set $p^{-1}(0)$. Null-rays are introduced in Definition \ref{d:defray}, they are paths tangent to a family of convex cones $\Gamma_m$ defined in Section \ref{s:gammam}.

\smallskip

For $m\in T^*(\R\times X)$ which is not in the double characteristic set $\{p=0\}\cap\{dp=0\}$, $\Gamma_m$ is simply the positive (or negative) cone generated by $H_p$ taken at point $m$, i.e., $\Gamma_m=\R^+\cdot H_p(m)$ (or $\Gamma_m=\R^-\cdot H_p(m)$, depending on whether $\tau\geq 0$ or $\tau\leq 0$). In the double characteristic set $\Sigma_{(2)}=\{p=0\}\cap\{dp=0\}\subset M$, the definition of the cones $\Gamma_m$ is more involved, and several formulas will be provided in Section \ref{s:thecones}. We can already say (see Appendix \ref{s:what}) that $\Gamma_m$ can be recovered as the convexification of the limits of all cones $\Gamma_{m_j}$ for $m_j\rightarrow m$ and $m_j\notin \Sigma_{(2)}$. 

\paragraph{Abnormal extremals.} Roughly speaking, Corollary \ref{c:singpropag}, Theorem \ref{t:inclusion} and Corollary \ref{c:xy} strengthen the idea that properties of general sub-Laplacians may be influenced not only by the geometry of null-bicharacteristics but also by the presence of abnormal extremal lifts of singular curves in the corresponding sub-Riemannian geometry; in other words, the latter curves play a role at the ``quantum level'' of general sub-Laplacians. 

\smallskip

This role was already foreboded in a particular case in the work of Richard Montgomery \cite{Mon95} about zero loci of magnetic fields, and it is central in the Treves conjecture about hypoelliptic analyticity (see \cite{treves1999symplectic} for the conjecture and \cite{albano2018analytic} for recent results). To the author's knowledge, it is the first result which illustrates this fact for \emph{general} sub-Laplacians.

\paragraph{Related literature.} As a particular case of Theorem \ref{t:singprop}, if $A$ is elliptic, then we recover Hörmander's result \cite[Proposition 3.5.1]{Ens} already mentioned above (see also \cite[Theorem 8.3.1 and Theorem 23.2.9]{Hor07} and \cite[Theorem 1.2.23]{Lerner}). In case $A$ has only double characteristics on a symplectic submanifold it was obtained in \cite{Mel84} in codimension $2$, and by B. and R. Lascar \cite{Las82}, \cite{LasLas82}  in the general case, using constructions of parametrices instead of positive commutator estimates as used in \cite{Mel86} (see also Remark \ref{r:biblio}). We also mention the paper \cite{Sav19} where a result about propagation of singularities in the case of ``quasi-contact'' sub-Laplacians is proved. The subelliptic wave propagator has also recently been studied in \cite{martini} to prove spectral multiplier estimates, but the construction is restricted to the ``elliptic part'' of the symbol where $a\neq 0$.

\smallskip

In \cite{Mel86}, two other results are proved, namely the finite speed of propagation for $P$ and an estimate on the heat kernel, but it is not our purpose to discuss here these other results, whose proofs are written in details in \cite{Mel86}. 

\paragraph{Open questions.} Here are a few natural questions that our work could help to answer:
\begin{itemize}
\item Is it possible to find explicitly the form of the wave kernel $K_G$ (and not only its singularities as in Theorem \ref{t:inclusion}), even for simple sub-Laplacians? This would generalize the Hadamard parametrix to subelliptic wave equations. The only known cases are apparently the Heisenberg case (see \cite{Nac82}, \cite{Tay86}, \cite{GKH02}), the contact case (\cite{Las82}, \cite{LasLas82}, \cite{Mel84}) and the quasi-contact case \cite{Sav19}.
\item The answer to the above question could pave the way towards a better understanding of the {\it subelliptic heat kernel} in the presence of abnormal extremals, thanks to the ``transmutation formula'' sometimes attributed to Y. Kannai \cite{Kan77} (see also \cite{cheeger}). Indeed, as proved in \cite{Lea87}, the subelliptic heat kernel $p_t(x,y)$ enjoys the asymptotics $2t{\rm Log}(p_t(x,y))\rightarrow d(x,y)^2$ as $t\rightarrow 0$ where $d(x,y)$ is the sub-Riemannian distance (see also \cite{jerison}), but the refinement of this limit into a full asymptotic development of $p_t(x,y)$ is known only in the absence of abnormal extremals minimizing the distance (see \cite{Ben88}). Note that subelliptic heat kernels have been studied a lot, see for instance \cite{neel} for one of the last major achievements in this field.
\item Is it possible to prove an analogue of Egorov's theorem in the framework of sub-Riemannian geometry? The usual formulation of Egorov theorem is that the evolution of a quantum observable ${\rm Op}_{\hbar}(b)$ is analogous to the evolution of the corresponding classical observable: 
$$
U^{-t}_H {\rm Op}_{\hbar}(b)U^t_H={\rm Op}_{\hbar}(b\circ \varphi_H^t)+O_t(\hbar), \quad \text{where} \quad U^t_H=e^{-\frac{it\hat{H}_h}{h}}, \ \ \hat{H}_h=-\frac{\hbar^2}{2m}\Delta
$$
and $\varphi^t_H$ is the Hamiltonian flow associated to the principal symbol $H$ of $\Delta$ (see for example \cite[Theorem IV-10]{Rob87}). Analogues of Egorov's theorem in the subelliptic framework have been derived for instance in \cite{colin} and \cite{fermanian}, but only in cases where there are no abnormal extremals. Would abnormal extremals play a role in a general subelliptic version of Egorov's theorem, as in the present work?
\item Is it possible to design a physical experiment with electrons in a magnetic field which would illustrate the propagation along singular curves as in Corollary \ref{c:singpropag}?\footnote{Thanks to R. Montgomery for suggesting this question.} Indeed, certain singular curves appear naturally as zero loci of magnetic fields, and high-frequency quantum particles tend to concentrate on such curves, see \cite{Mon95}. R. Montgomery says in \cite{Mon95} that ``singular curves persist upon quantization''.
\end{itemize}

\subsection{Organization of the paper} 
The goal of this work is to provide a fully detailed proof of Theorem \ref{t:singprop}, Corollary \ref{c:singpropag}, Theorem \ref{t:inclusion} and Corollary \ref{c:xy} and to explain how these results are related to sub-Riemannian geometry.

\smallskip

 In Section \ref{s:thecones}, we define the convex cones $\Gamma_m$ generalizing bicharacteristics and give explicit formulas for them, then prove their semi-continuity with respect to $m$, and finally introduce null-rays and ``time functions''. These functions are by definition non-increasing along the cones $\Gamma_m$. In this section, there is no operator, we work at a purely ``classical'' level.
 
 \smallskip 

The proof of Theorem \ref{t:singprop} is based on a positive commutator argument: the idea, which dates back at least to \cite{Ens} (see also \cite[Chapter I.2]{Ivr19}), is to derive an \emph{energy inequality} from the computation of a quantity of the form $\text{Im}(Pu,Lu)$, where $L$ is some well-chosen (pseudodifferential) operator. In Section \ref{s:poscommutator}, we compute this quantity for $L=\Op(\Phi)D_t$ where $\Phi$ is a time function, we write it under the form $\frac12(Cu,u)$ for an explicit second-order operator $C$ which, up to remainder terms, has non-positive symbol.

\smallskip

In Section \ref{s:proof}, we derive from this computation the sought energy inequality, which in turn implies Theorem \ref{t:singprop}. This proof requires to construct specific time functions and to use the powerful Fefferman-Phong inequality \cite{FP78}.

\smallskip

In Section \ref{s:xyproof}, we prove Corollary \ref{c:singpropag}. This requires to explain basic concepts of sub-Riemannian geometry, notably we define normal geodesics, singular curves, and abnormal extremal lifts.

\smallskip  

In Section \ref{s:th1}, we prove Theorem \ref{t:inclusion}: the main idea is to see $K_G$ itself as the solution of a subelliptic wave equation. We also prove Corollary \ref{c:xy} in the same section.

\smallskip

The reader will find in Appendix \ref{a:techn} the sign conventions for symplectic geometry that we use throughout this note, and a short reminder on pseudodifferential operators in Appendix \ref{a:pseudo}. Finally, in Appendix \ref{s:what}, we explain how the cones $\Gamma_m$ can be defined in a unified way as Clarke generalized gradients, thus making a bridge between our computations and Clarke's version of Pontryagin's maximum principle.

\paragraph{Acknowledgments.} I am very grateful to Yves Colin de Verdière, for his help at all stages of this work. Several ideas, notably in Sections \ref{s:th1}, are due to him. I also thank him for having first showed me R. Melrose's paper and for his constant support along this project, together with Emmanuel Trélat. I am also thankful to Andrei Agrachev, Richard Lascar and Nicolas Lerner for very interesting discussions related to this paper. Finally, many thanks are due to the referee whose suggestions improved the readability of the paper.

\section{The cones $\Gamma_m$} \label{s:thecones}
At double characteristic points where in particular $dp=0$, the Hamiltonian vector field $H_p$ vanishes, and the usual propagation of singularities result \cite[Theorem 6.1.1]{DH} recalled in Section \ref{s:motivations} does not provide any information. In \cite{Mel86}, R. Melrose defines convex cones $\Gamma_m$ which replace the usual propagation cone $\R^+\cdot H_p$ at these points, and which indicate the directions in which singularities of the subelliptic wave equation \eqref{e:subwave} may propagate.  The cones $\Gamma_m$ can be equivalently defined
\begin{itemize}
\item symplectically (Section \ref{s:gammam}) - this is the most synthetic definition;
\item with analytic formulas  (Section \ref{s:formula}) - very useful for proofs;
\item with the Clarke generalized gradient (Appendix \ref{s:what}) - which gives a clear geometric insight, although it is not useful in our proofs.
\end{itemize}
The cones $\Gamma_m$ have particularly simple expressions in case $A$ is a sub-Laplacian as in Definition \ref{e:sL}. These expressions are given along the proof of Corollary \ref{c:singpropag} in Section \ref{s:proofcorollary}, and also linked with the ``Clarke version'' of the Pontryagin maximum principle in Appendix \ref{s:what}.

\subsection{First definition of the cones $\Gamma_m$} \label{s:gammam}
In this section, we introduce several notations, and we define the cones $\Gamma_m$.

\smallskip

We consider $a\in C^\infty(T^*X)$ satisfying
\begin{equation} \label{e:homo}
a(x,\xi)\geq 0, \quad a(x,r\xi)=r^2 a(x,\xi), \quad r>0
\end{equation} 
in canonical coordinates $(x,\xi)$. Also we consider 
$$p=\tau^2-a\in C^\infty(M), \qquad \text{where} \quad M=T^*(\R\times X)\setminus 0.$$ In the end, $a$ and $p$ will be the principal symbols of the operators $A$ and $P$ introduced in Section \ref{s:intro}, but for the moment we work at a purely classical level and forget about operators. In a nutshell, at points where $p=0$ and $dp=0$, the cones $\Gamma_m$ are defined thanks to the Hessian of $p$.

\smallskip

We set
\begin{equation*}
M_+=\{m\in M, \ p(m)\geq 0, \tau\geq 0\}, \qquad M_-=\{m\in M, \ p(m)\geq 0, \tau\leq 0\};
\end{equation*}
in particular, $M_+\cup M_-=\{p\geq0\}$. Let
\begin{equation*}
\Sigma=\{m\in M; \ p(m)=0, \ \tau\geq 0\}.
\end{equation*}

\paragraph{The cones $\Gamma_m$ for $m\in M_+$.}
For $m\in M_+$, we consider the set
\begin{equation*}
\mathscr{H}_m=\R^+\cdot H_p(m)\subset T_mM,
\end{equation*}
where $H_p$ is the Hamiltonian vector field of $p$ verifying $\omega(H_p,Z)=-dp(Z)$ for any smooth vector field $Z$. In this formula and in all the sequel, $\omega$ is the canonical symplectic form on the cotangent bundle $M$. 

\smallskip

We note that
$$
p(m)\geq 0, \ dp(m)=0 \ \Rightarrow \ \mathscr{H}_m=\{0\}
$$ 
where $dp(m)$ stands for the differential of $p$ taken at point $m$.
We therefore extend the notion of ``bicharacteristic direction'' at such $m$. This will be done first for $m\in M_+$, then also for $m\in M_-$, but never for $m\in \{p<0\}$: the cones $\Gamma_m$ are not defined for points $m\in\{p<0\}$.

\smallskip

Let 
\begin{equation*}
\Sigma_{(2)}=\{m\in M, \tau=a=0\} \subset \Sigma.
\end{equation*}
Note that 
$$\Sigma_{(2)}=M_+\cap M_-=\{m\in M; \ \tau=a=p=0,\ da=dp=0\}$$
due to the the positivity \eqref{e:homo}. Thus, for $m\in\Sigma_{(2)}$, the Hessian of $a$ is well-defined: it is a quadratic form on $T_mM$. We denote by $a_m=\frac12 \Hess a$ \emph{the half of this Hessian}, and by $$p_m=(d\tau)^2-a_m$$ the half of the Hessian of $p$. For $m\in\Sigma_{(2)}$, we set
\begin{equation} \label{e:deflambdam}
\Lambda_m=\{w\in T_mM; \ d\tau(w)\geq 0, \ p_m(w)\geq0\}
\end{equation}
and, still for $m\in\Sigma_{(2)}$,
\begin{equation} \label{e:defgammam}
\Gamma_m:=\{v\in T_mM; \ \omega(v,w)\leq 0 \ \ \forall w\in \Lambda_m\}.
\end{equation}
If $m\in M_+\setminus \Sigma_{(2)}$, we set 
\begin{equation} \label{e:GammaHm}
\Gamma_m=\mathscr{H}_m.
\end{equation}
In particular, the cones $\Gamma_m$ are defined also at points $m$ outside $\Sigma$, i.e. for which $p(m)\neq 0$. Note also that the relation \eqref{e:GammaHm} says that the cones $\Gamma_m$ are \emph{positive} cones.

\paragraph{The cones $\Gamma_m$ for $m\in M_-$.} In order to extend the definition of the cones $\Gamma_m$ to $M_-$, we want this extension to be consistent with the previous definition at points in $M_+\cap M_-=\Sigma_{(2)}$. We observe that $M_-$ is the image of $M_+$ under the involution sending $\tau$ to $-\tau$. For $(t,\tau,\alpha)\in M_-$, we set
\begin{equation*}
\Gamma_m=\Gamma_{m'} \qquad \text{where  }\quad m'=(t,-\tau,\alpha)\in M_+.
\end{equation*}
It is clear that at points of $M_+\cap M_-=\Sigma_{(2)}$, the two definitions of $\Gamma_m$ coincide. With this definition in $M_-$, note that for $m\in M_-\setminus \Sigma_{(2)}$, there is a sign change:
\begin{equation} \label{e:-}    
\Gamma_m=-\mathscr{H}_m.
\end{equation}

\smallskip

In summary, the formulas \eqref{e:defgammam}, \eqref{e:GammaHm} and \eqref{e:-} define $\Gamma_m$ at any point $m\in M_+\cup M_-$, with different definitions for $m\in\Sigma_{(2)}$, $m\in M_+\setminus \Sigma_{(2)}$ and $m\in M_-\setminus \Sigma_{(2)}$.
The cones $\Gamma_m$ are not defined for $m\notin M_+\cup M_-$. For any $m\in M_+\cup M_-$, the cone $\Gamma_m$ is closed and convex.

\subsection{Formulas for the cones $\Gamma_m$} \label{s:formula}

In this section, we derive a formula for the cones $\Gamma_m$ when $m\in\Sigma_{(2)}$ which is more explicit than \eqref{e:defgammam}. It relies on the computation of the polar of a cone defined by a non-negative quadratic form:
\begin{proposition} \label{p:polar}
Let $Q$ be a non-negative quadratic form on a real vector space $V$, and let $$\Theta=({\rm ker}(Q))^\perp\subset V^*$$ where $\perp$ is understood in the duality sense and $V^*$ is the topological dual of $V$. Let $$\Lambda=\left\{\xi=(\xi_0,\eta)\in\R\times V; \ \xi_0\geq Q(\eta)^{\frac12}\right\}$$ and $$\Lambda^0=\left\{\xi'\in(\R\times V)^*; \ \forall \xi\in\Lambda, \ \xi'(\xi)\leq0\right\}.$$
Then 
\begin{equation} \label{e:lambda0degenerate}
\Lambda^0=\left\{\xi'=(\xi_0',\eta')\in(\R\times V)^*;\ \eta'\in \Theta \text{  and  } -\xi_0'\geq (Q^*(\eta'))^{\frac12}\right\}
\end{equation}
where $\R^*$ is identified with $\R$ and
\begin{equation} \label{e:S*}
Q^*(\eta')=\sup_{\eta\notin {\rm ker}(Q)} \frac{\eta'(\eta)^2}{Q(\eta)}.
\end{equation}
\end{proposition}
\begin{proof}
Let $\xi'=(\xi_0',\eta')\in (\R\times V)^*$ such that $\eta'\in\Theta$ and $-\xi_0'\geq(Q^*(\eta'))^{\frac12}$, we seek to prove that $\xi'\in\Lambda^0$. Let $\xi=(\xi_0,\eta)\in\Lambda$. In particular, $\xi_0\geq (Q(\eta))^{\frac12}$. We have 
\begin{equation*}
\xi'(\xi)=\xi_0'(\xi_0)+\eta'(\eta)\leq -(Q^*(\eta'))^{\frac12}(Q(\eta))^{\frac12}+\eta'(\eta) \leq 0
\end{equation*}
hence $\xi'\in \Lambda^0$, which proves one inclusion. 

\smallskip

Conversely, to prove that $\Lambda^0$ is included in the expression \eqref{e:lambda0degenerate}, we first note that if $\eta'\notin \Theta$, then $(\xi_0',\eta')\notin \Lambda^0$ for any $\xi_0'\in \R^*$. Indeed, if $\eta'\notin \Theta$, there exists $\eta\in V$ such that $Q(\eta)=0$ and $\eta'(\eta)>0$. Thus, considering $\xi=(0,\eta)$, which is in $\Lambda$ by assumption, we get $\xi'(\xi)=\eta'(\eta)>0$ for any $\xi_0'\in \R^*$ and $\xi'=(\xi_0',\eta')$, proving that $\xi'\notin \Lambda^0$. Now, if $\xi'=(\xi_0',\eta')\in \Lambda^0$ with $\eta'\in \Theta$, we take $\xi_n=(\xi_{0n},\eta_n)$ with $\eta_n\notin {\rm ker}(Q)$ so that $\eta'(\eta_n)^2/Q(\eta_n)\rightarrow Q^*(\eta')$, and $\eta'(\eta_n)\geq 0$ and $\xi_{0n}=Q(\eta_n)^{\frac12}$. Then $\xi_n\in\Lambda$. Therefore, $\xi'(\xi_n)\leq 0$, which implies that $-\xi_0'\geq (Q^{*}(\eta'))^{\frac12}$. This proves the result.
\end{proof}
Applying the previous proposition to $Q=a_m$ yields a different definition of the cones $\Gamma_m$. First, $\Lambda_m$, which has been defined in \eqref{e:deflambdam}, can be written as
\begin{equation*}
\Lambda_m=\left\{w\in T_mM; \ d\tau(w)\geq (a_m(w))^{\frac12}\right\}.
\end{equation*}
Since the definition of $\Lambda_m$ does not involve $dt$, we have $v(\partial_t)=0$ for any $v\in \Lambda_m^0$. Now, using the notation $a_m^*$ to denote \eqref{e:S*} when $Q=a_m$, Proposition \ref{p:polar} yields that 
\begin{align*}
&\qquad \quad \Lambda_m^0=\R^+(-d\tau+B_0), \\
&B_0=\left\{b_0\in  ({\rm ker}(a_m))^\perp, \ a_m^*(b_0)\leq 1\right\}.
\end{align*}
The duality $\perp$ is computed with respect to the space ${\rm ker}(a_m)\subset T(T^*X)$, i.e., $b_0\in T^*(T^*X)$. 

\smallskip

Here, we consider $a_m$ as a quadratic form on $T(T^*X)$ instead of $M=T(T^*(\R\times X))$. This is also related to the fact that the Hessian $a_m$ depends only on the projection $\pi_2(m)$, where $\pi_2:M\rightarrow T^*X$ is the canonical projection on the second factor, and not on really on the other components of $m$. \emph{These slight abuse of notations cause no problem, and will thus be repeated several times in the sequel.}

\smallskip

Comparing the definition of $\Lambda_m^0$ as the polar cone of $\Lambda_m$ and the definition \eqref{e:defgammam} of $\Gamma_m$, we see that $\Gamma_m$ is exactly the image of $\Lambda_m^0$ through the canonical isomorphism $\omega(v,\cdot)\mapsto v$ between $T_m^*M$ and $T_mM$. Thus, 
\begin{equation} \label{e:gammam2} 
\begin{split}
& \qquad \qquad \Gamma_m=\R^+(\partial_t+B ), \\
& B=\left\{b \in {\rm ker}(a_m)^{\perp_{\omega_{X}}}, \ a_m^*(\mathcal{I}(b))\leq 1\right\}.
\end{split}
\end{equation}
Here, $\perp_{\omega_{X}}$ designates the symplectic orthogonal with respect to the canonical symplectic form $\omega_{X}$ on $T^*X$ and 
\begin{equation} \label{e:Iisom}
\begin{split}
\mathcal{I}:T(T^*X)&\rightarrow T^*(T^*X),\\ b &\mapsto \omega_{X}(b ,\cdot)
\end{split}
\end{equation}
 is the canonical isomorphism between $T(T^*X)$ and $T^*(T^*X)$.  Formula \eqref{e:gammam2} plays a key role in the sequel. An equivalent formula in terms of the so-called ``fundamental matrix'' associated to $a_m$ is derived in Appendix \ref{s:what}.

\medskip

\textbf{Important remark.} In case $A=\sum_{i=1}^K Y_i^*Y_i$ is a sub-Laplacian as in Definition \ref{e:sL}, the expression $a_m^*(\mathcal{I}(b ))$ which appears in \eqref{e:gammam2} is equal to $g(d\pi(b))$ where $\pi:T^*X\rightarrow X$ is the canonical projection, and $g$ is the sub-Riemannian metric associated to the vector fields $Y_i$ (see Lemma \ref{l:eqamg}).

\subsection{Inner semi-continuity of the cones $\Gamma_m$} \label{s:innerwithformula}
Using the formula \eqref{e:gammam2}, we can prove a continuity property for the cones $\Gamma_m$.
\begin{lemma} \label{l:inner} Let $a\in C^\infty(T^*X)$ satisfying \eqref{e:homo}.
The assignment $m\mapsto \Gamma_m$ is inner semi-continuous on $M_+\cup M_-=\{p\geq 0\}$. In other words, if both conditions
\begin{enumerate}[(i)]
\item $m_j\in M_+\cup M_-$ for any $j\in\N$, and $m_j\rightarrow  m$ as $j\rightarrow +\infty$;
\item $v_j\in \Gamma_{m_j}$ for any $j\in\N$, and $v_j\rightarrow v\in T_mM$ as $j\rightarrow +\infty$,
\end{enumerate}
hold, then $v\in \Gamma_m$.
\end{lemma}
Before proving Lemma \ref{l:inner}, let us explain the intuition behind this semi-continuity. Recall that the cones $\Gamma_m$ generalize bicharacteristic directions at points where $\tau=a=p=0$ and $da=dp=0$. To define the cones $\Gamma_m$ at these points, following formulas \eqref{e:deflambdam} and \eqref{e:defgammam},  we have first considered directions where $p$ grows (since $p=0$ and $dp=0$, we consider the (half) Hessian $p_m$), yielding $\Lambda_m$, and then $\Gamma_m$ has been defined as the (symplectic) polar cone of $\Lambda_m$. This is exactly parallel to a procedure which yields bicharacteristic directions in the non-degenerate case: the directions along which $p$ grows, verifing $dp(v)\geq 0$, form a cone, and it is not difficult to check that its (symplectic) polar consists of a single direction given by the Hamiltonian vector field of $p$. This is a unified vision of the cones $\Gamma_m$, in the sense that they are obtained in a unified way, no matter whether $m\in\Sigma_{(2)}$ or not. The proof of Lemma \ref{l:inner} we give below is however purely analytic, and does not use this geometric intuition. Note that Appendix \ref{s:what} provides still another unified vision of the cones $\Gamma_m$, thanks to the notion of Clarke generalized gradients.

\medskip

\emph{Proof of Lemma \ref{l:inner}.} The assignments $$\Sigma_{(2)}\ni m\mapsto \Gamma_m \ \ \text{and} \ \ M_+\cup M_- \setminus \Sigma_{(2)}\ni m\mapsto \Gamma_m$$ are clearly continuous thanks to formula \eqref{e:defgammam} (resp. \eqref{e:GammaHm} and \eqref{e:-}). Therefore, we restrict to the case where $m\in\Sigma_{(2)}$ and $m_j\in M_+\cup M_- \setminus \Sigma_{(2)}$.

\smallskip

According to \eqref{e:GammaHm} and \eqref{e:-}, the cone $\Gamma_{m_j}$ at $m_j=(t_j,\tau_j,x_j,\xi_j)$ is given by:
\begin{equation} \label{e:simpleGammamj}
\Gamma_{m_j}={\rm sgn}(\tau_j)\R^+[2\tau_j\partial_t-H_a(m_j)]
\end{equation}
where $H_a(m_j)$ is the Hamiltonian vector field of $a$ at $m_j$. Dividing by $2\tau_j\neq 0$, we rewrite it as
\begin{equation} \label{e:Gammamj}
\Gamma_{m_j}=\R^+\left(\partial_t-\frac12\frac{H_a(m_j)}{\tau_j}\right)
\end{equation}
If $a(m_j)=0$, then $H_a(m_j)=0$ since $a$ is a quadratic form. Thus $\Gamma_{m_j}=\R^+\partial_t$, thus any limit of elements of $\Gamma_{m_j}$ is contained in $\Gamma_m$ according to \eqref{e:gammam2} (take $b=0$). We thus assume that $a(m_j)\neq 0$.

\smallskip

In the sequel, we work in a chart near $m$ and $\|\cdot\|_{M}$ denotes a norm in this chart. We recall that $a_m$ is half the Hessian of $a$ at $m$, thus a bilinear form on $T_mM$. When its two arguments are $m_j-m$ (which we can view as an element of $T_mM$ now that we are working in a chart), it is written $a_m(m_j-m)$. Also, in the sequel the notation $o$ accounts for the $j\rightarrow +\infty$ limit.

\smallskip

We distinguish two cases.

\smallskip

Firstly, if $a_m(m_j-m)=o(\|m_j-m\|_M^2)$, then using a Taylor expansion and $a(m)=da(m)=0$, we get $a(m_j)^{\frac12}=o(\|m_j-m\|_M)$. Since $a$ is a quadratic form, this implies that $d(a^{\frac12})(m_j)=o(1)$, where the notation in the left-hand side stands for the differential of $a^{\frac12}$ taken at point $m_j$. In turn, we obtain $$\frac12\frac{da(m_j)}{\tau_j}=\frac{a(m_j)^{\frac12}}{\tau_j}d(a^{\frac12})(m_j)=o(1)$$
since $|\tau_j|\geq a(m_j)^{\frac12}$ due to $p(m_j)\geq 0$. This implies that $H_a(m_j)/\tau_j\rightarrow 0$ and plugging into \eqref{e:Gammamj}, we conclude that the limiting directions of vectors $v_j$ of $\Gamma_{m_j}$ as $j\rightarrow +\infty$ belong to $\R^+\partial_t$ and thus to $\Gamma_m$.

\smallskip

Secondly, if $a_m(m_j-m)$ is not $o(\|m_j-m\|_M^2)$, then we use the following lemma.
\begin{lemma} \label{l:linalg} 
If $a_m(m_j-m)$ is not $o(\|m_j-m\|_M^2)$, then for any $v\in T_mM$, there holds $$\frac12\frac{da(m_j)(v)}{a(m_j)^{1/2}}=\frac{a_m(m_j-m,v)}{a_m(m_j-m)^{1/2}}+o(1).$$
The notation $a_m(m_j-m,v)$ means that we evaluate the bilinear form $a_m$ at $(m_j-m,v)$.
\end{lemma}
\begin{proof}
In a chart, we combine the two expansions
\begin{align*}
da(m_j)(v)&=2a_m(m_j-m,v)+o(\|m_j-m\|_M) \\
a(m_j)&=a_m(m_j-m)+o(\|m_j-m\|_M^2)
\end{align*}
to get the result.
\end{proof}

In view of \eqref{e:Gammamj} and \eqref{e:gammam2}, the inner semi-continuity at $m$ is equivalent to proving that 
\begin{equation} \label{e:ineqam*}
a_m^*\left(\frac12\frac{a(m_j)^{\frac12}}{\tau_j}\frac{da(m_j)}{a(m_j)^{\frac12}}\right)\leq 1+o(1).
\end{equation}
Using $|\tau_j|\geq a(m_j)^{\frac12}$ and Lemma \ref{l:linalg}, for any $v\in T_mM\setminus {\rm ker}(a_m)$, there holds
\begin{equation*}
\frac{1}{a_m(v)}\left(\frac12\frac{a(m_j)^{\frac12}}{\tau_j}\frac{da(m_j)(v)}{a(m_j)^{\frac12}}\right)^2\leq \frac{a_m(m_j-m,v)^2}{a_m(v)a_m(m_j-m)}+o(1)\leq 1+o(1)
\end{equation*}
by Cauchy-Schwarz. Hence, by definition of $a_m^*$ (see \eqref{e:S*}), \eqref{e:ineqam*} holds, which concludes the proof of Lemma \ref{l:inner}.

\subsection{Null-rays and time functions} \label{s:time}

 We now define null-rays, which are the integral curves of the cone field $\Gamma_m$. They appear in the statement of Theorem \ref{t:singprop}, and they play an important role in the present work. 
 
 In this definition, we use the following notation: given a Lipschitz curve $\gamma:I\rightarrow M_+$ defined on some interval $I\subset \R$, the set-valued derivative $\gamma'(s)$ for $s\in I$ is the set of all tangent vectors $X\in T_{\gamma(s)}M$ such that there exists $s_n\rightarrow s$ with $s_n\neq s$ for any $n\in \N$, verifying that $\forall f\in C^\infty(M)$, $\frac{f(\gamma(s_n))-f(\gamma(s))}{s_n-s}\rightarrow Xf$.
\begin{definition} \label{d:defray}
A forward-pointing ray for $p$ is a Lipschitz curve $\gamma:I\rightarrow M_+$ defined on some interval $I\subset \R$ with (set-valued) derivative $\gamma'(s)\subset \Gamma_{\gamma(s)}$ for all $s\in I$. Such a ray is forward-null if $\gamma(s)\in\Sigma$ for any $s\in I$. We define backward-pointing rays similarly, with $\gamma$ valued in $M_-$, and backward-null rays, with $\gamma$ valued in $\{m\in M; \ p(m)=0, \ \tau\leq 0\}$.
\end{definition}
Under the terminology ``\emph{ray}'', we mean either a forward-pointing or a backward-pointing ray; under the terminology ``\emph{null-ray}'', we mean either a forward-null or a backward-null ray. 

\smallskip

In particular null-rays live in $\{p=0\}$.

\smallskip

Fixing a norm $|\cdot|$ on $TM$, the expression \eqref{e:gammam2} implies that near any point $m\in M_+$, there is a (locally) uniform constant $c>0$ such that 
\begin{equation} \label{e:tislarge}
v\in \Gamma_m \Rightarrow v=T\partial_t + v', \qquad |v'|\leq cT
\end{equation}
where $v'$ is tangent to $T^*X$. Thus, if $\gamma:I\rightarrow M_+$ is a forward-pointing ray (thus a Lipschitz curve) defined for $s\in I$, \eqref{e:tislarge} implies that $dt/ds\geq c'|d\gamma/ds|$, hence $$\frac{d\gamma}{dt}=\frac{\ \frac{d\gamma}{ds}\ }{\ \frac{dt}{ds}\ }$$ is well-defined (possibly set-valued), i.e., $\gamma$ can be parametrized by $t$.

\smallskip

We define the {\it length} of a ray $\gamma:s\in[s_0,s_1]\rightarrow M_+$ by $$\ell(\gamma):=|t(s_1)-t(s_0)|.$$

\begin{remark}\label{r:nullray}
Thanks to the above parametrization and with a slight abuse in the terminology, we say that there is a null-ray of length $|T|$ from $(y,\eta)$ to $(x,\xi)$ if there exists a null-ray (in the sense of Definition \ref{d:defray}) parametrized by $t$ which joins $(0,\tau,y,\eta)$ to $(T,\tau,x,\xi)$, where $\tau$ verifies $\tau^2=a(y,\eta)=a(x,\xi)$.
\end{remark}

Time functions, which we now introduce, are one of the key ingredients of the proof of Theorem \ref{t:singprop}.
\begin{definition} \label{d:timefunction}
A $C^\infty$ function $\phi$ near $\overline{m}\in\{p\geq 0\}\subset M$ is a time function near $\overline{m}$ if in some neighborhood $N$ of $\overline{m}$,
\begin{equation*}
\phi \text{ is non-increasing along $\Gamma_m$, $m\in N\cap \{p\geq 0\}$}.
\end{equation*}
In particular, $\phi$ is non-increasing along the Hamiltonian vector field $H_p$ in $M_+$ but non-decreasing along $H_p$ in $M_-$ (due to \eqref{e:-}).
\end{definition}

Note that outside $\{p\geq 0\}$, there is no constraint on the values of $\phi$. The proof of Theorem \ref{t:singprop} relies on a positive commutator technique (Section \ref{s:poscommutator}) applied with a particular time function (Section \ref{s:proof}).
%
%
%
%
%

\section{A positive commutator} \label{s:poscommutator}
The proof of Theorem \ref{t:singprop} is based on a ``positive commutator'' technique, also known as ``multiplier'' or ``energy'' method in the literature. The idea is to derive an inequality from the computation of a quantity of the form $\text{Im}(Pu,Lu)$ where $L$ is some well-chosen (pseudodifferential) operator. In the present work, the operator $L$ is related to the time functions introduced in Definition \ref{d:timefunction}.

\smallskip

In the sequel, we use polyhomogeneous symbols, denoted by $S^m_{\rm phg}$, and the \emph{Weyl quantization}, denoted by $\Op:S^m_{\rm phg}\rightarrow \Psi^m_{\rm phg}$ (see Appendix \ref{a:pseudo}). 
 For example, we consider the operator $D_t=\frac{1}{i}\partial_t=\Op(\tau)$ (of order $1$). The operator $A\in \Psi^2_{\rm phg}$ has principal symbol $a\in C^\infty(T^*X)$ satisfying \eqref{e:homo}, and $P=D_t^2-A$ has principal symbol $p=\tau^2-a$.
 
 \smallskip

Also, $\Phi(t,x,\xi)$ designates a smooth \textit{real-valued} function on $M$, homogeneous of degree $\alpha\in\R$ in $\xi$, compactly supported on the base $\R\times X$, and independent of $\tau$. In Section \ref{s:proof}, we will take $\Phi$ to be a time function. By the properties of the Weyl quantization, $\Op(\Phi)$ is a compactly supported selfadjoint (with respect to $\nu$)  pseudodifferential operator of order $\alpha$.

\smallskip

Our goal in Section \ref{s:OpC} will be to compute $C$ defined by\footnote{In \cite{Mel86}, $C$ is explicitly defined as $\text{Im} (\Op(\Phi) D_tu,Pu):=(Cu,u)$; however the formulas (6.1) and (6.2) in \cite{Mel86} are not coherent with this definition, but they are correct if we take the definition \eqref{e:mainequality} for $C$.}
\begin{equation} \label{e:mainequality}
\text{Im} (Pu, \Op(\Phi) D_tu):=\frac12(Cu,u),
\end{equation}
since this will allow us to derive the inequality \eqref{e:propsing} which is the main ingredient in the proof of Theorem \ref{t:singprop}.

\subsection{The operator $C$} \label{s:OpC}
Our goal in this section is to compute $C$ defined by \eqref{e:mainequality}. 
\begin{lemma}
We have
\begin{equation} \label{e:melrose6.1}
C=D_t\Op(\Phi_t') D_t-\frac{i}{2}([A,\Op(\Phi)]D_t+D_t[A,\Op(\Phi)])+\frac12(A\Op(\Phi_t')+\Op(\Phi_t')A).
\end{equation} 
where $\Phi_t'=\partial_t \Phi$.
\end{lemma}
Note that $C$ is of order $2+\alpha$, although we could have expected order $3+\alpha$ by looking too quickly at \eqref{e:mainequality}.
\begin{proof}
We have
\begin{equation} \label{e:Im}
\text{Im} (Pu, \Op(\Phi) D_t u)=I_1-I_2
\end{equation}
with 
\begin{equation*}
I_1=\text{Im} (D_t^2u, \Op(\Phi) D_t u) \qquad \text{and} \qquad I_2=\text{Im} (Au, \Op(\Phi) D_t u).
\end{equation*}

\smallskip

Noticing that
\begin{equation*} 
[D_t,\Op(\Phi)]=\Op(\frac{1}{i}\Phi_t')
\end{equation*}
(see \cite[Theorem 4.6]{Zwo}), we have for $I_1$:
\begin{align}
I_1&=\frac{1}{2i}\left((D_t^2u,\Op(\Phi) D_tu)-(\Op(\Phi) D_t u, D_t^2u)\right)\nonumber\\
&=\frac{1}{2i}\left((D_t\Op(\Phi) D_t^2u,u)-(D_t^2\Op(\Phi) D_t u, u)\right)\nonumber\\
&=-\frac{1}{2i}(D_t[D_t,\Op(\Phi)] D_t u, u)\nonumber \\
&=-\frac{1}{2i}(D_t \frac{1}{i}\Op(\Phi_t') D_t u, u)\nonumber\\
&=\frac{1}{2}(D_t\Op(\Phi_t') D_t u, u) \label{e:I1}
\end{align}
 Then, we write $\Op(\Phi)D_t=S+iT$ where
\begin{align}
S&=\frac{1}{2}(\Op(\Phi)D_t+D_t\Op(\Phi)) \nonumber\\ 
T&=\frac{1}{2i}(\Op(\Phi)D_t-D_t\Op(\Phi))=\frac12\Op(\Phi'_t). \label{e:T}
\end{align}
Using that $A$, $S$ and $T$ are selfadjoint, we compute $I_2$:
\begin{align}
I_2&=\text{Im}(Au,(S+iT)u)=\text{Im}((S-iT)Au,u)=\frac{1}{2i}([S,A]u,u)-\text{Re}((TAu,u))\nonumber\\
&=\frac{1}{2i}([S,A]u,u)-\frac12((TA+AT)u,u).\label{e:I2}
\end{align}
Furthermore,
\begin{equation} \label{e:SA}
[S,A]=\frac12([\Op(\Phi),A]D_t+D_t[\Op(\Phi),A]).
\end{equation}
All in all, combining \eqref{e:Im}, \eqref{e:I1}, \eqref{e:T}, \eqref{e:I2} and \eqref{e:SA}, we find that $C$ in \eqref{e:mainequality} is given by \eqref{e:melrose6.1}.
\end{proof}

\subsection{The principal and subprincipal symbols of $C$}
In this section, we compute the operator $C$ modulo a remainder term in $\Psi^\alpha_{\phg}$.  All symbols and pseudodifferential operators used in the computations are polyhomogeneous (see Appendix \ref{a:pseudo}); we denote by $\sigma_p(C),\sigma_{\text{sub}}(C)$ the principal symbol of $C$ and its sub-principal symbol.
We use the Weyl quantization, denoted by $\Op$, in the variables $y=(t,x)$, $\eta=(\tau,\xi)$, hence we have for any $b\in S_{\rm phg}^m$ and $c\in S_{\rm phg}^{m'}$:
\begin{equation} \label{e:composition}
\Op(b) \Op(c)-\Op(bc+\frac{1}{2i}\{b,c\})\in \Psi_{\rm phg}^{m+m'-2}
\end{equation}
and 
\begin{equation} \label{e:composition2}
[\Op(b),\Op(c)]-\Op(\frac{1}{i}\{b,c\})\in  \Psi_{\rm phg}^{m+m'-3}.
\end{equation}
Note that in \eqref{e:composition2}, the remainder is in $\Psi_{\rm phg}^{m+m'-3}$, and not only in $\Psi_{\rm phg}^{m+m'-2}$ (see \cite[Theorem 18.5.4]{Hor07}, \cite[Theorem 4.12]{Zwo}). Finally, we recall that $\Phi(t,x,\xi)$ is homogeneous in $\xi$ of degree $\alpha$. 

\begin{lemma}
There holds 
\begin{equation}\label{e:sigma2C}
\sigma_p(C)=\tau H_p \Phi-\Phi_t'p
\end{equation}
and
\begin{equation}\label{e:sbp}
\sigma_{\text{sub}}(C)=0.
\end{equation}
\end{lemma}
\begin{proof}
We compute each of the terms in \eqref{e:melrose6.1} modulo $ \Psi^\alpha_{\rm phg}$. We prove the following formulas:
\begin{align}
\frac12(A\Op(\Phi_t')+\Op(\Phi_t')A)&=\Op(a\Phi_t')\quad \text{mod } \Psi^\alpha_{\rm phg}\label{e:C3}\\
D_t\Op(\Phi_t')D_t&=\Op(\tau^2\Phi_t') \quad \text{mod } \Psi^\alpha_{\rm phg} \label{e:C1}\\
\frac{i}{2}([A,\Op(\Phi)]D_t+D_t[A,\Op(\Phi)])&=\Op(\tau\{a,\Phi\}) \quad \text{mod } \Psi^\alpha_{\rm phg} \label{e:C2}
\end{align}

Firstly, \eqref{e:C3} follows from the fact that $A=\Op(a)\text{ mod } \Psi^0_{\rm phg}$ (since the subprincipal symbol of $A$ vanishes) and from \eqref{e:composition} applied once with $b=a$, $c=\Phi_t'$, and another time with $b=\Phi_t'$ and $c=a$.

\smallskip

Secondly, $$\Op(\Phi_t')D_t=\Op(\Phi_t')\Op(\tau)=\Op(\Phi_t'\tau+\frac{1}{2i}\{\Phi_t',\tau\})+\Psi_{\rm phg}^{\alpha-1}$$ thanks to \eqref{e:composition}. Hence, using again \eqref{e:composition}, we get 
\begin{align*}
D_t\Op(\Phi_t')D_t&=\Op(\tau)\Op(\Phi_t'\tau+\frac{1}{2i}\{\Phi_t',\tau\}) \quad \text{mod } \Psi_{\rm phg}^{\alpha}\\
&=\Op(\tau^2\Phi_t'+\frac{\tau}{2i}\{\Phi_t',\tau\}+\frac{1}{2i}\{\tau,\Phi_t'\tau\}) \quad \text{mod } \Psi_{\rm phg}^\alpha
\end{align*}
which proves \eqref{e:C1}.

\smallskip

Thirdly, thanks to $A=\Op(a)\text{ mod } \Psi^0_{\rm phg}$ and \eqref{e:composition2}, we have
\begin{align*}
[A,\Op(\Phi)]&=\Op\left(\frac{1}{i}\{a,\Phi\}\right) \quad \text{mod } \Psi^{\alpha-1}_{\rm phg}
\end{align*}
(note that the remainder is in $\Psi^{\alpha-1}_{\rm phg}$, not in $\Psi^{\alpha}_{\rm phg}$). Using \eqref{e:composition}, we get
\begin{equation*}
[A,\Op(\Phi)]D_t+D_t[A,\Op(\Phi)]=\Op\left(\frac{2\tau}{i}\{a,\Phi\}\right)\quad \text{mod } \Psi^\alpha_{\rm phg}
\end{equation*}
which proves \eqref{e:C2}.

\smallskip

In particular, we get the principal symbol
\begin{align*}
\sigma_p(C)=\tau^2\Phi_t'-\tau H_a \Phi+\Phi_t'a.
\end{align*}
Using $p=\tau^2-a$, we can write it differently:
\begin{align}
\sigma_p(C)&=\tau^2\Phi_t'-\tau\{\tau^2-p,\Phi\}+\Phi_t'a\nonumber\\
&=\tau^2\Phi_t'-\tau\{\tau^2,\Phi\}+\tau H_p \Phi+\Phi_t'a\nonumber\\
&=\tau^2\Phi_t'-2\tau^2\Phi_t'+\tau H_p \Phi+\Phi_t'a\nonumber\\
&=\tau H_p \Phi-\Phi_t'p. \nonumber
\end{align}
Moreover, the formulas \eqref{e:C3}, \eqref{e:C1} and \eqref{e:C2} imply that the subprincipal symbol of $C$ vanishes, which concludes the proof.
\end{proof}

\section{Proof of Theorem \ref{t:singprop}} \label{s:proof}
 The goal of this section is to prove Theorem \ref{t:singprop}. For $V\subset T^*X$ and $I\subset\R$, we set
\begin{equation} \label{e:defSt}
\begin{split}
\mathscr{S}(I; V)=&\{(s,y,\eta)\in I\times T^*X, \ \text{there exist } (x,\xi)\in V,\ \tau\in \R  \text{ and a ray }\\
&\qquad \qquad  \text{from } (s,\tau,y,\eta) \text{ to } (0,\tau,x,\xi)\}.
\end{split}
\end{equation}
Most of the time, we will consider $I\subset (-\infty,0)$. Also, when $I$ or $V$ is reduced to a single element, for example $I=\{t\}$ or $V=\{(x,\xi)\}$, we will simplify the notations by dropping out the brackets in the notation: for example, instead of $\mathscr{S}(\{t\},V)$, we write $\mathscr{S}(t,V)$. Finally, take care that the above notation \eqref{e:defSt} refers to rays, and not null-rays (see Definition \ref{d:defray}).

\smallskip

 With the above notations, Theorem \ref{t:singprop} can be reformulated as follows: for any $t>0$ and any $(x_0,\xi_0)\in WF(u(0))$, there exists $(y_0,\eta_0)\in WF(u(-t))\cup WF(\partial_tu(-t))$ such that $(-t,y_0,\eta_0)\in \mathscr{S}(-t; (x_0,\xi_0))$ and one of the rays from $(y_0,\eta_0)$ to $(x_0,\xi_0)$ is null.
 
 \paragraph{First reduction of the problem.}
If $a(x_0,\xi_0)\neq 0$, then Theorem \ref{t:singprop} follows from the usual propagation of singularities theorem \cite[Theorem 6.1.1]{DH} and the fact that $\Gamma_m=\R^{\pm}\cdot H_p(m)$ for $m\notin\Sigma_{(2)}$. Therefore, in the sequel we assume that $a(x_0,\xi_0)=0$. 

\smallskip

Also, note that, to prove Theorem \ref{t:singprop}, it is sufficient to find $T>0$ independent of $(x,\xi)$ (and possibly small) such that the result holds for any $t\in (0,T)$.

\paragraph{Idea of the proof of Theorem \ref{t:singprop}.} 
To show Theorem \ref{t:singprop}, we will prove for $T>0$ sufficiently small an inequality of the form
\begin{equation} \label{e:propsingularites}
\|\Op(\Psi_0) u\|_{H^s}^2\leq c(\|\Op(\Psi_0) u\|_{L^2}^2+\|\Op(\Psi_1)u\|_{L^2}^2)+\text{ Remainder terms}
\end{equation}
where $\Psi_0$ and $\Psi_1$ are functions of $t,x,\xi$ such that
\begin{itemize}
\item the function $\Psi_0$ is supported near $t\in [-T,0]$ and the function $\Psi_1$ near $t=-T$;
\item on their respective supports in $t$, the operators $\Op(\Psi_0)$ and $\Op(\Psi_1)$  microlocalize respectively near $(x_0,\xi_0)$ and $\mathscr{S}(-T; (x_0,\xi_0))$.
\end{itemize}
 Then, assuming that $u$ is smooth on the support of $\Psi_1$, we deduce by applying \eqref{e:propsingularites} for different functions $\Psi_0$ with different degrees of homogeneity in $\xi$ that $u$ is smooth on the support of $\Psi_0$.

\smallskip

The inequality \eqref{e:propsingularites}, written more precisely as \eqref{e:propsing} below, will be proved by constructing a time function $\Phi(t,x,\xi)$ such that the time derivative $\Phi_t'$ is equal to $\Phi_t'=\Psi_1^2-\Psi_0^2$, and then by applying the Fefferman-Phong inequality to the operator $C$ given by \eqref{e:melrose6.1} (for this $\Phi$). 

\paragraph{Reduction to $\R^d$.} Let us explain how to reduce Theorem \ref{t:singprop} to a problem in $\R^d$. We first notice that it is sufficient to prove Theorem \ref{t:singprop} locally, i.e., only for the restriction of $u$ to some small open subset $U\subset X$. This follows from the two following facts:
\begin{itemize}
\item firstly, by the finite speed of propagation of subelliptic wave equations (proved in \cite[Section 3]{Mel86}), for any small $t>0$, the value of the solution of  \eqref{e:subwave} at time $0$ at $x\in X$ only depends on its values in a small neighborhood of $x$ at time $-t$. 
\item secondly, null-rays stay close from their departure points for short times. This follows from \eqref{e:GammaHm}, \eqref{e:-}, \eqref{e:gammam2}. 
\end{itemize}
Theorem \ref{t:singprop} is thus a \emph{local} statement for \emph{short} times.
As a consequence, a first reduction for its proof consists in fixing a small open subset $U\subset X$ and addressing the propagation of singularities problem in short time for the restriction of $u$ to $U$. 

\smallskip

 Then, we consider a coordinate chart $\psi:U\rightarrow \Omega\subset \R^d$. The differential operator $A$ is pushed forward by $\psi$ into a differential operator $\widetilde{A}$ on $\R^d$ which is also real, second-order, self-adjoint, non-negative and subelliptic. Moreover, we can lift $\psi$ to a symplectic mapping $\psi_{\text{lift}}:(x,\xi)\mapsto (\psi(x),((d_x\psi(x))^{-1})^T\xi)$. Through the differential of $\psi_{\text{lift}}$, the cones $\Gamma_m$ (computed with $a=\sigma_P(A)$, in $X$) are sent to the same cones, computed this time with $\widetilde{a}=\sigma_P(\widetilde{A})$ in $\R^d$. This follows from the ``symplectic'' definition of the cones in Section \ref{s:gammam} and the fact that $\sigma_P(\widetilde{A})$ is the pushforward of $\sigma_P(A)$. Hence, $\psi_{\text{lift}}$ maps also null-rays to null-rays. As a consequence, if we prove Theorem \ref{t:singprop} in $\Omega$, then pulling back the situation to $U\subset X$ proves Theorem \ref{t:singprop} in full generality. 
 
 \smallskip
 
\emph{In the sequel, we thus work in $\Omega\subset \R^d$.}

\subsection{Construction of the time function} \label{s:cstrtimefunction}

As explained in the introduction of this section, we construct a time function $\Phi(t,x,\xi)$ which verifies several properties. A time function is also constructed in the classical proof of Hörmander's propagation of singularities theorem \cite[Proposition 3.5.1]{Ens}, but in the present context of subelliptic wave equations, the construction is  more involved since the cones $\Gamma_m$ along which time functions should be non-increasing contain much more than a single direction (compare \eqref{e:GammaHm} with \eqref{e:gammam2}). The following lemma summarizes the properties that the time functions we need thereafter should satisfy. The figures below, notably Figure \ref{f:touchdown}, may help to understand the statement and its proof.

\begin{lemma} \label{l:cstrtime}
Let $(x_0,\xi_0)\in T^*\Omega\setminus 0$ and $V\subset V'$ be sufficiently small open conic (in $\xi$) neighborhoods of $(x_0,\xi_0)$ such that $\overline{V}\subset V'$. There exists $T>0$ such that for any $0\leq \delta_0\leq T/10$ and any $\alpha\geq 0$, there exists a smooth function $\Phi(t,x,\xi)$ with the following properties:
\begin{enumerate}[(1)]
\item it is compactly supported in $t,x$;
\item it is homogeneous of degree $\alpha$ in $\xi$;
\item it is independent of $\tau$;
\item there exists $\delta>0$ such that at any point of $M$ where $p\geq -2\delta a$, there holds $\tau H_p\Phi\leq 0$.
\item its derivative in $t$ can be written $\Phi_t'=\Psi_1^2-\Psi_0^2$ with $\Psi_0$ and $\Psi_1$ homogeneous of degree $\alpha/2$ in $\xi$;
\item $\Psi_0=0$ outside $\mathscr{S}((-T,\delta_0); V')$ and $\Psi_1=0$ outside $\mathscr{S}((-T-\delta_0,-T+\delta_0); V')$;
\item $\Psi_0>0$ on $\mathscr{S}((-T+\delta_0,0); V)$;
\item $\Phi$ is a time function outside $\mathscr{S}((-T-\delta_0,-T+\delta_0);V')$.
\end{enumerate}
\end{lemma}
All of the above properties of $\Phi$ will be used in Sections \ref{s:decompoc} and \ref{s:endprf} to prove Theorem \ref{t:singprop}. The rest of Section \ref{s:cstrtimefunction} is devoted to the proof of Lemma \ref{l:cstrtime}.

\smallskip

We fix $(x_0,\xi_0)\in T^*\Omega\setminus0$. As said in the introduction of Section \ref{s:proof}, we assume  that $a(x_0,\xi_0)=0$, and we set $\overline{m}=(0,0,x_0,\xi_0)\in \Sigma_{(2)}$ where the first two coordinates correspond to the variables $t,\tau$. For $m$ near $\overline{m}$, the cone $-\Gamma_m$ is the cone with base point $m$ and containing the opposite of the directions of $\Gamma_m$.

\smallskip

We are looking for a $\tau$-independent time function; since any ray lives in a slice $\tau=\text{const.}$ (see \eqref{e:GammaHm}, \eqref{e:-} and \eqref{e:gammam2}), we first construct $\Phi$ in the slice $\tau=0$, and then we extend $\Phi$ to any $\tau$ so that it does not depend on $\tau$.  If we start from a time function in $\{\tau=0\}$, then its extension is also a time function: indeed, the image of a ray contained in $\{\tau\neq 0,a=0\}$ under the map $\tau\mapsto 0$ is also a ray, this follows from the fact that $\R^+\partial_t\subset \Gamma_m$ for any $m\in\Sigma_{(2)}$ (see \eqref{e:gammam2}). Thus, the property of being non-increasing along $\Gamma_m$ is preserved under this extension process. Thus, in the sequel, we work in $\{\tau=0\}$ and do not care about (3).

\smallskip

We now explain why it is natural to impose condition (2) on time functions. Indeed, there is a global homogeneity in $\xi$ of the cones $\Gamma_m$ and consequently of the null-rays:

\smallskip

\noindent\textbf{Homogeneity Property.} If $[T_1,T_2]\ni t\mapsto\gamma(t)=(x(t),\xi(t))\in \{a=0\}$ is a null-ray parametrized by $t$, then for any $\lambda>0$, $[T_1,T_2]\ni t\mapsto \gamma_\lambda(t)=(x(t),\lambda\xi(t))$ is a null-ray parametrized by $t$. Note that $\gamma(t)$ and $\gamma_\lambda(t)$ have the same projection on $X$ for any $t\in [T_1,T_2]$.

\smallskip

This property, illustrated in Figure \ref{f:cones}, follows from \eqref{e:defgammam}. It will be helpful to find $\Phi$ satisfying Point (2) in Lemma \ref{l:cstrtime}. 
\begin{figure}[h!]
     \centering
     \captionsetup{justification=centering}
         \includegraphics[width=7.5cm]{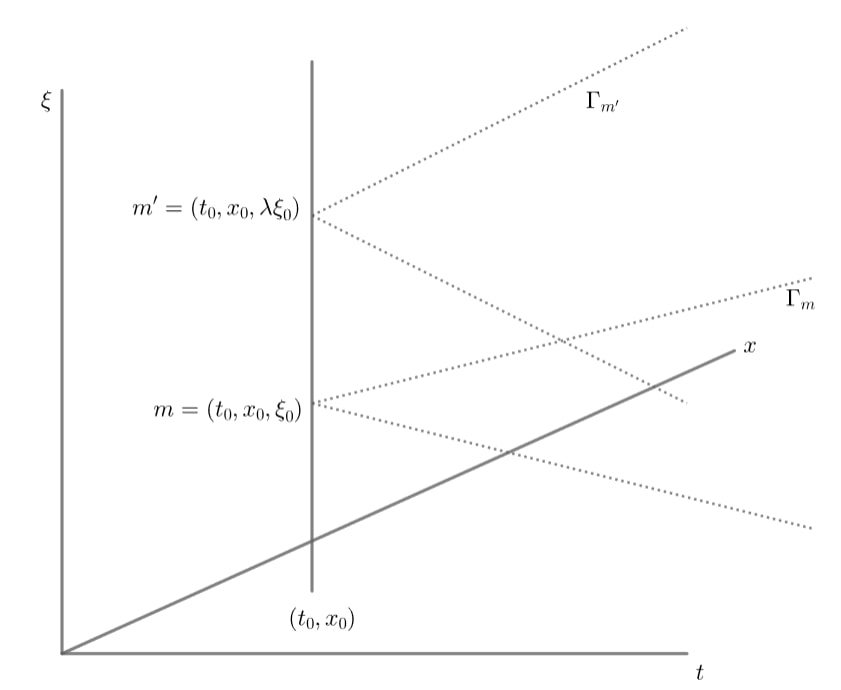}
\caption{The coordinates and the cones $\Gamma_m$. On the picture, the cone $\Gamma_{m'}$ has an aperture which is equal to $\lambda$ times the aperture of $\Gamma_m$ (this is ``homogeneity'').}
       \label{f:cones}
\end{figure}

\smallskip

 At this point we should say that since we are working in the slice $\{\tau=0\}$, we will use in the sequel the following convenient abuse of notations: for $m=(t,0,x,\xi)\in T^*\R\times T^*\Omega$, we still denote by $m$ the projection of $m$ on $\R\times T^*\Omega$ obtained by throwing away the coordinate $\tau=0$. The fact that the whole picture is now embedded in $\R^{2d+1}$ (see Figure \ref{f:cones}) is very convenient: for example, after throwing away the coordinate $\tau=0$, we see the cones $\Gamma_m$ as subcones of $\R^{2d+1}$ (and not of its tangent space).

\smallskip

Also, in the sequel, we only consider points for which $t\geq -T$ for some (small) $T>0$. We take  $0\leq\delta_0\leq T/10$.

  \smallskip
  
  We will define $\Phi$ so that along any ray ending near $(t=0,x_0,\xi_0)$, its profile looks like Figure \ref{f:touchdown}, which is the standard picture for time functions (see \cite[Figure E.2]{DZ}):
  \begin{figure}[h!]
\centering
  \includegraphics[width=12cm]{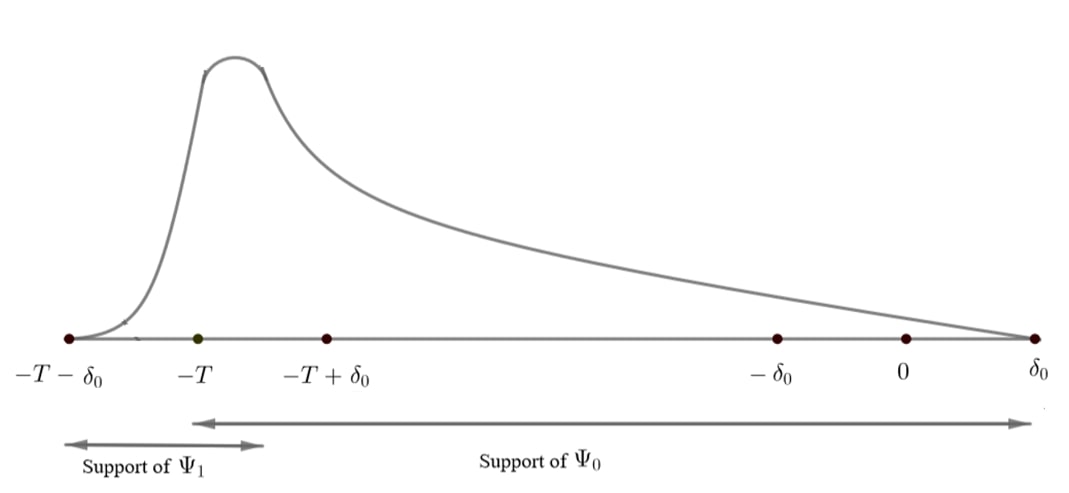}
  \captionsetup{justification=centering}
         \caption{Profile of the function $\Phi$ along a ray ending near $(t=0,x_0,\xi_0)$.\\ The abscissa is this ray, parametrized by $t$.}
         \label{f:touchdown}
\end{figure}

The support of $\Phi$ is a conic set of points (due to homogeneity in $\xi$), namely $\mathscr{S}([-T-\delta_0, \delta_0];V')$. To describe the construction of $\Phi$, we will go backwards in time, from right to left in Figure \ref{f:touchdown}, thus looking for $\Phi$ {\it increasing} along {\it backward} rays (at least up to time $-T+\delta_0$). Recall that backward rays are just integral curves of the field of cones $-\Gamma_m$. All the rays we consider in the sequel are parametrized by time.

\smallskip

The first important point is that rays enjoy closedness and continuity properties:
  \begin{lemma} \label{l:Sinnercontinuous}
\begin{enumerate}
\item For any closed $V\subset T^*\Omega$ and any $T\geq 0$, the set $\mathscr{S}(-T;V)$ is closed. 
\item The mapping $(T,x,\xi)\mapsto \mathscr{S}(-T;(x,\xi))$ is inner semi-continuous, meaning that when $(T_n,x_n,\xi_n)\rightarrow (T,x,\xi)$, any point obtained as a limit, as $n\rightarrow +\infty$, of points of $\mathscr{S}(-T_n; (x_n,\xi_n))$  belongs to $\mathscr{S}(-T;(x,\xi))$.
\end{enumerate}
\end{lemma}
\begin{proof}
Both properties follow from the locally uniform Lipschitz continuity \eqref{e:tislarge} combined with the extraction of Lipschitz rays as in the Arzel\`a-Ascoli theorem and the fact that the cones $\Gamma_m$ are closed, convex and inner semi-continuous (Lemma \ref{l:inner})
\end{proof}
This lemma implies that in the statement of Lemma \ref{l:cstrtime}, the set $\mathscr{S}((-T+\delta_0,0);V)$ is just slightly larger than $\mathscr{S}((-T+\delta_0,0);(x_0,\xi_0))$ (and similarly for other sets $\mathscr{S}(\cdot,\cdot)$ involved in the statement of Lemma \ref{l:cstrtime}).

\smallskip

The second important point is the following. We recall that we look for $\Phi$ homogeneous of degree $\alpha\geq 0$ in $\xi$; in particular it is increasing in the fibers. To guarantee simultaneously this homogeneity and the fact that $\Phi$ increases along backward rays, we have to consider the ``worst rays'', namely backward rays which are ``diving'' towards smaller $|\xi|$: we impose that even along the rays (parametrized by time) diving most quickly towards small $|\xi|$, $\Phi$ increases. Then, by homogeneity, $\Phi$ is also increasing along any other ray.

\smallskip

We apply this procedure for defining $\Phi$ to backward rays emanating from a fixed closed conic neighborhood $V$ of $(x_0,\xi_0)$. Point 2 of Lemma \ref{l:Sinnercontinuous} implies that 
\begin{equation} \label{e:V}
\begin{split}
&\text{if $V$ is a sufficiently small neighborhood of $(x_0,\xi_0)$,}\\
&\text{$\Phi$ is strictly increasing from time $0$ to time $-T+\delta_0$}\\
\text{along }&\text{any backward-pointing ray starting from any point $(x,\xi)\in V$.}
\end{split}
\end{equation}
We choose $V'$ a sufficiently small neighborhood of $V$, we can impose that $\Phi$ vanishes at any $(t,x,\xi)\notin \mathscr{S}((-T+\delta_0,\delta_0);V')$ for $t\in (-T+\delta_0,\delta_0)$.
 
 \smallskip
 
 Up to now, our construction defines $\Phi$ only for times satisfying $-T+\delta_0\leq t \leq \delta_0$. To complete the constrution, we extend it arbitrarily (but smoothly) in $\mathscr{S}([-T-\delta_0,-T+\delta_0],V')$ so that it vanishes for $t\leq -T-\delta_0$ (refer again to Figure \ref{f:touchdown}).
 
 \smallskip

Note that with our construction, $\Phi$ is decreasing along any backward ray, even along backward rays coming from the region $\{\Phi=0\}$ and entering $\{\Phi\neq 0\}$, and not just along those emanating from $(0,x_0,\xi_0)$ or a nearby point. The only place where the behaviour of $\Phi$ is not controlled is $\mathscr{S}([-T-\delta_0,-T+\delta_0],V')$, and this is why $\Phi$ is a time function on $\mathscr{S}([-T+\delta_0,\delta_0],V')$ and not on the whole larger set $\mathscr{S}([-T-\delta_0,\delta_0],V')$ (see Property (8)).

\smallskip

For $t\geq -T+\delta_0$, we have $\Phi_t'\leq 0$ since $\partial_t\in \Gamma_m$, and thus we set $\Psi_0=\sqrt{-\Phi_t'}$. Then, following the rays backwards in time, we make $\Psi_0$ fall to $0$ between times $-T+\delta_0$ and $-T$ (see Figure \ref{f:touchdown}). Similarly, following the rays backward from time $-T+\delta_0$ to time $-T-\delta_0$, we extend $\Phi$ smoothly and homogeneously (in the fibers in $\xi$) in a way that $\Phi$ is compactly supported in the time-interval $(-T-\delta_0, \delta_0)$ and $\Phi_t'+\Psi_0^2\geq 0$. Finally, we set $\Psi_1=\sqrt{\Phi_t'+\Psi_0^2}$. 

\smallskip 

In Lemma \ref{l:cstrtime}, Properties (1), (2), (3), (5), (6), (8) follow from the construction. Property (7) follows from \eqref{e:V}. 

\smallskip

Finally, let us explain why Property (4) holds. It follows from our construction that for some $\delta>0$, $\Phi$ is decreasing along rays computed with respect to
\begin{equation*}
\begin{split}
& \qquad \qquad \qquad \widetilde{\Gamma}_m=\R^+(\partial_t+\widetilde{B} ), \\
& \widetilde{B}=\left\{b \in {\rm ker}(a_m)^{\perp_{\omega_{X}}}, \ a_m^*(\mathcal{I}(b))\leq (1-2\delta)^{-1/2}\right\}.
\end{split}
\end{equation*}
instead of $\Gamma_m$ (i.e., integral curves of the field of cones $\widetilde{\Gamma}_m$); note that $\widetilde{\Gamma}_m$ is strictly larger than $\Gamma_m$ due to \eqref{e:gammam2}. In fact, for $\delta>0$ sufficiently small (depending on $V$), these new rays are included in $\mathscr{S}((-T+\delta_0,\delta_0);V)$. Then, in view of the proof of Lemma \ref{l:inner}, and notably \eqref{e:ineqam*}, this gives Property (4).


\subsection{A decomposition of $C$} \label{s:decompoc}

When $\Phi$ satisfies (2), (3), (4) and (5) in Lemma \ref{l:cstrtime}, the operator $C$ given by \eqref{e:melrose6.1} can be expressed as follows:
\begin{proposition} \label{p:decompoC}
If $\Phi$ satisfies (2), (3), (4) and (5) in Lemma \ref{l:cstrtime}, then writing $\Phi_t'=\Psi_1^2-\Psi_0^2$, there holds
\begin{equation} \label{e:decompoc}
C=C'+R+R'P+PR'-\delta(\Op(\Psi_0) A\Op(\Psi_0)+D_t\Op(\Psi_0)^2D_t)
\end{equation}
 where $\delta>0$ is the same as in (4), $$R'=-\frac{\delta}{2}\Op(\Phi_t')\in \Psi^\alpha_{\phg}, \qquad R=\delta\Op(\Psi_1)(D_t^2+A)\Op(\Psi_1)\in \Psi^{2+\alpha}_{\phg},$$ and $C'\in \Psi^{2+\alpha}_{\phg}$ has non-positive principal symbol and vanishing subprincipal symbol.
\end{proposition}

 We start the proof of this proposition with the following improved (and corrected) version of \cite[Lemma 5.3]{Mel86}:
\begin{lemma} \label{l:falselemma}
Let $\phi$ be a time function near $\overline{m}\in\Sigma_{(2)}$ which does not depend on $\tau$ and such that
\begin{equation}\label{e:strongertimefunction}
 \tau H_p\phi\leq 0 \qquad \text{on} \quad \{p\geq -2\delta a\}.
\end{equation}
Then there holds
\begin{equation} \label{e:ineqHp}
 \tau H_p\phi\leq \phi_t'(p+ 2\delta a)
\end{equation}
in a neighborhood of $\overline{m}$.
\end{lemma}
Note that for any time function, the inequality \eqref{e:strongertimefunction} holds on the smaller set $\{p\geq 0\}$. Assuming \eqref{e:strongertimefunction} is a stronger requirement.

\begin{proof}[Proof of Lemma \ref{l:falselemma}]
Since $\phi$ does not depend on $\tau$, we know that $q=\tau \{p,\phi\}$ is a quadratic polynomial in $\tau$, vanishing at $\tau=0$:
\begin{equation*}
q=b\tau^2-c\tau, \quad p=\tau^2-a, \quad a\geq 0.
\end{equation*}
More explicitly, $b=2\phi_t'$ and $c=\{a,\phi\}$.
From \eqref{e:strongertimefunction}, we know that $b\tau^2-c\tau\leq 0$ for $\tau$ sufficiently large, hence $b\leq 0$. Moreover, \eqref{e:strongertimefunction} also implies that if $b=0$, then $c=0$, hence $\phi_t'=H_p\phi=0$, and \eqref{e:ineqHp} is automatically satisfied. Otherwise, $b<0$. Since $q\leq 0$ on $\tau\notin [-((1-2\delta)a)^{1/2},((1-2\delta)a)^{1/2}]$ by \eqref{e:strongertimefunction}, we get that the other zero of $q$, $\tau=c/b$, must lie in $[-((1-2\delta)a)^{1/2},((1-2\delta)a)^{1/2}]$. Thus, $c^2\leq b^2a(1-2\delta)$. Then, 
\begin{equation*}
\tau\{p,\phi\}-\phi_t'p=\frac12 b(\tau-c/b)^2+(b^2a-c^2)/2b\leq  ba\delta=2\phi_t'a\delta
\end{equation*}
where we used that $b<0$. 
\end{proof}

\begin{proof}[Proof of Proposition \ref{p:decompoC}]
Setting $r'=-\frac{\delta}{2}\Phi_t'$, we have according to Lemma \ref{l:falselemma} with $\phi=\Phi$:
\begin{equation} \label{e:ineqsymbols}
\tau\{p,\Phi\}-\Phi_t'p-2r'p\leq 2\Phi_t'a\delta+\Phi_t'p\delta=\Phi_t'\delta(\tau^2+a)=\delta (\Psi_1^2-\Psi_0^2)(\tau^2+a).
\end{equation}
We set $R=\delta\Op(\Psi_1)(D_t^2+A)\Op(\Psi_1)$ and $R'=\Op(r')$. It follows from \eqref{e:ineqsymbols}, \eqref{e:sigma2C}, \eqref{e:sbp} and \eqref{e:composition} that the operator 
\begin{equation}\label{e:C'}
C'=C-R-(R'P+PR')+\delta(\Op(\Psi_0) A\Op(\Psi_0)+D_t\Op(\Psi_0)^2D_t)
\end{equation}
has non-positive principal symbol and vanishing subprincipal symbol. This proves Proposition \ref{p:decompoC}.
\end{proof}

%

\subsection{The Fefferman-Phong inequality}

The Fefferman-Phong inequality \cite{FP78} (see also \cite[Section 2.5.3]{Lerner}) can be stated as follows: for any pseudodifferential operator $C_1'$ of order $2+\alpha$ whose (Weyl) symbol is non-positive, there holds for any $u\in C_c^\infty(\R^n)$, 
\begin{equation} \label{e:fpineq}
(C_1'u,u)_{L^2(\R^n)}\leq c(({\rm Id}-\Delta)^{\alpha/2}u,u)_{L^2(\R^n)}
\end{equation}
where $\Delta$ is a Riemannian Laplacian on $\R^n$.
The following lemma is a simple microlocalization of this inequality. The definition of the essential support, denoted by ${\rm essupp}$, is recalled in Appendix \ref{a:pseudo}.

\begin{lemma} \label{l:feffphong}
Let $\alpha\geq 0$, and let $W,W'\subset T^*(\R\times \Omega)$ be conic sets such that $W'$ is a conic neighborhood of $W$. Let $C'\in \Psi_{\phg}^{2+\alpha}$ with ${\rm essupp}(C')\subset W$ such that $\sigma_p(C')\leq 0$ and $\sigma_{\text{sub}}(C')\leq 0$. Then there exists $C_\alpha\in\Psi^{\alpha/2}_{\rm phg}$ with ${\rm essupp}(C_\alpha)\subset W'$ such that
\begin{equation} \label{e:modfpineq}
\forall u \in C_c^\infty(\R\times \Omega), \qquad (C'u,u)_{L^2}\leq c(\|C_\alpha u\|_{L^2}^2+\|u\|_{L^2}^2).
\end{equation}
\end{lemma}
\begin{proof} Taking a microlocal cut-off $\chi$ homogeneous of order $0$, essentially supported in $W'$ and equal to $1$ on a conic neighborhood of $W$, we see that
\begin{align}
(C'u,u)&=(C'(\Op(\chi)+\Op(1-\chi))u,(\Op(\chi)+\Op(1-\chi))u)\nonumber\\
&=(\Op(\chi)C'\Op(\chi)u,u)+(Q'u,u) \label{e:split1}
\end{align}
where $Q'\in\Psi^{-\infty}$ is explicit:
\begin{equation*}
Q'=\Op(1-\chi)C'\Op(\chi)+\Op(\chi)C'\Op(1-\chi)+\Op(1-\chi)C'\Op(1-\chi).
\end{equation*}
Since $Q'\in\Psi^{-\infty}$, we have in particular
\begin{equation} \label{e:split4}
(Q'u,u)\leq c\|u\|^2_{L^2}.
\end{equation}
Then, we write $C'=C_1'+C_2'$ where $C_1'$ has non-positive full Weyl symbol, and $C_2'\in\Psi_{\rm phg}^\alpha$. First, we apply \eqref{e:fpineq} with $\Op(\chi)u$ instead of $u$: we obtain 
\begin{equation}\label{e:split2}
(\Op(\chi)C_1'\Op(\chi)u,u)\leq c\|C_\alpha u\|_{L^2}^2
\end{equation}
with $C_\alpha=(\Id-\Delta)^{\alpha/4}\Op(\chi)$. Secondly, writing $C_2'=(\Id-\Delta)^{\alpha/4}C_2''(\Id-\Delta)^{\alpha/4}$ with $C_2''\in \Psi_{\rm phg}^0$, we see that 
\begin{equation}\label{e:split3}
(\Op(\chi)C_2'\Op(\chi)u,u)\leq c\|C_\alpha u\|_{L^2}^2.
\end{equation}
Combining \eqref{e:split1}, \eqref{e:split4}, \eqref{e:split2} and \eqref{e:split3}, we get \eqref{e:modfpineq}.
 \end{proof}

\subsection{End of the proof of Theorem \ref{t:singprop}} \label{s:endprf}

We come back to the proof of Theorem \ref{t:singprop}. We fix $(x_0,\xi_0)\in T^*\Omega\setminus 0$ and consider $u$ a solution of \eqref{e:subwave}. For the moment, we assume that $u$ is \emph{smooth}. We consider a time function $\Phi$ as constructed in Lemma \ref{l:cstrtime}. 

\smallskip

%


Using \eqref{e:decompoc}, we have
\begin{align*}
0&=2\text{Im} (Pu, \Op(\Phi) D_tu)\\
&=(Cu,u)\\
&=((C'+R+R'P+PR'-\delta(\Op(\Psi_0) A\Op(\Psi_0) +  D_t\Op(\Psi_0)^2D_t ))u,u).
\end{align*}
Hence, using $Pu=0$ and applying Lemma \ref{l:feffphong} to $C'$, we get:
\begin{align*}
(A\Op(\Psi_0)  u,\Op(\Psi_0)  u)+\|\Op(\Psi_0) D_t u\|_{L^2}^2&\leq  c((R_\alpha+R'P+PR'+C')u,u)\\
&\leq c_\alpha (\|C_\alpha u\|_{L^2}^2+\|u\|_{L^2}^2+(R_\alpha u,u)).
\end{align*}
with $c_\alpha\geq 1/\delta$ and $R_\alpha=R$, just to keep in mind in the forthcoming inequalities that it depends on $\alpha$.

\smallskip

But $(A\Op(\Psi_0)  u,\Op(\Psi_0) u)\geq \frac1c ((-\Delta)^s \Op(\Psi_0) u,\Op(\Psi_0) u)-\|\Op(\Psi_0)  u\|^2$ by subellipticity \eqref{e:subelliptic}. Hence
\begin{equation} \label{e:propsing}
\|(-\Delta)^{s/2} \Op(\Psi_0)  u\|_{L^2}^2+\|\Op(\Psi_0)D_t  u\|_{L^2}^2\leq c_\alpha(\|C_\alpha u\|_{L^2}^2+\|u\|_{L^2}^2 +(R_\alpha u,u)+\|\Op(\Psi_0)  u\|_{L^2}^2)
\end{equation}
which we decompose into
\begin{equation}\label{e:propsing1}
\|(-\Delta)^{s/2} \Op(\Psi_0)  u\|_{L^2}^2\leq c_\alpha(\|C_\alpha u\|_{L^2}^2+\|u\|_{L^2}^2+(R_\alpha u,u)+\|\Op(\Psi_0)  u\|_{L^2}^2)
\end{equation}
and
\begin{equation}
\|\Op(\Psi_0)D_t  u\|_{L^2}^2\leq c_\alpha(\|C_\alpha u\|_{L^2}^2+\|u\|_{L^2}^2+(R_\alpha u,u)+\|\Op(\Psi_0)  u\|_{L^2}^2). \label{e:propsing2}
\end{equation}

Now, assume that $u$ is a general solution of \eqref{e:subwave}, not necessarily smooth. We have $u\in C^0(\R;\mathscr{D}(A^{1/2}))\cap C^1(\R;L^2(\Omega))$. We recall that $\Omega$ is a subset of $\R^d$. We have the following definition.
\begin{definition}
Let $s_0\in\R$ and $f\in\mathcal{D}'(\Omega)$. We shall say that $f$ is $H^{s_0}$ at $(x,\xi)\in T^*\Omega\setminus 0$ if there exists a conic neighborhood $W$ of $(x,\xi)$ such that for any $0$-th order pseudodifferential operator $B$ with ${\rm essupp}(B)\subset W$, we have $Bf\in H^s_{\text{loc}}(\Omega)$. \\
We shall say that $f$ is smooth at $(x,\xi)$ if it is $H^{s_0}$ at $(x,\xi)$ for any $s_0\in\R$.\\
When we say that $u$ is $H^{s_0}$ at $(t,y,\eta)$, we mean that $u(t)$ is $H^{s_0}$ at $(y,\eta)\in T^*\Omega$.
\end{definition}
\begin{lemma} \label{e:propagregu}
Let $V,V'$ be sufficiently small open conic neighborhoods of $(x_0,\xi_0)$ such that $\overline{V}\subset V'$. Let $u$ be a solution of \eqref{e:subwave}. If $u$ and $\partial_t u$ are smooth in $\mathscr{S}((-T-\delta_0,-T+\delta_0); V')$, then $u$ is smooth in
$$
U=\mathscr{S}((-T+\delta_0,0); V).
$$ 
\end{lemma}

\begin{proof}[Proof of Lemma \ref{e:propagregu}]
We set $u_\e=\rho_\e * u$ where $\rho_\e=\e^{-(d+1)}\rho(\cdot/\e)$ and $\rho\in C_c^\infty(\R^{d+1})$ is of integral $1$ (and depends on the variables $t,x$). Recall that $d$ is the dimension of $\Omega$.

\smallskip

Applying Lemma \ref{l:cstrtime} for any $\alpha\geq 0$ yields a function $\Phi_\alpha$ which is in particular homogeneous of degree $\alpha$ in $\xi$; its derivative in $t$ can be written $\Phi_\alpha'=(\Psi_{1}^\alpha)^2-(\Psi_{0}^\alpha)^2$ (the upper index being not an exponent). Then we apply \eqref{e:propsing1} to $u_\e$ and with $\alpha=0$: we get
\begin{equation} \label{e:equwitheps}
\|(-\Delta)^{s/2} \Op(\Psi_{0}^0)  u_\e\|_{L^2}^2\leq c_0(\|C_0u_\e\|_{L^2}^2+\|u_\e\|_{L^2}^2+(R_0u_\e,u_\e)+\|\Op(\Psi_{0}^0)  u_\e\|_{L^2}^2)
\end{equation}
where $R_0=\delta\Op(\Psi_{1}^0)(D_t^2+A)\Op(\Psi_{1}^0)$ (see Proposition \ref{p:decompoC}) and $c>0$ does not depend on $\e$.
All quantities
$$
\|C_0u\|_{L^2},\quad\|u\|_{L^2}, \quad (R_0u,u), \quad \|\Op(\Psi^0_0)  u\|_{L^2}^2
$$
are finite. Therefore, taking the limit $\e\rightarrow 0$ in \eqref{e:equwitheps}, we obtain $u\in H^{2s}$ in $U$. Using the family of inequalities \eqref{e:propsing1}, we can iterate this argument: first with $\alpha=2s$, then with $\alpha=4s, 6s$, etc, and each time we replace $\Psi_0^0$, $R_0$, $C_0$ by $\Psi_0^\alpha$, $R_\alpha$, $C_\alpha$. At step $k$, we deduce thanks to \eqref{e:propsing1} that $u\in H^{2ks}$. In particular, we use the fact that $\|C_\alpha u\|_{L^2}$ and $\|\Op(\Psi_0^\alpha)u\|_{L^2}$ are finite, which comes from the previous step of iteration since $C_\alpha$ is essentially supported close to the essential support of $C'$ (which is contained in the essential support of $\Phi$ thanks to \eqref{e:C'}). Thus, $u\in\bigcap_{k\in\N} H^{2ks}= C^\infty$ in $U$. 

\smallskip

Then, using \eqref{e:propsing2} for any $\alpha\in\N$ with $\Psi_0^{\alpha}$ in place of $\Psi_0$, we obtain that $D_tu$ is also $H^{\alpha}$ in $U$. Since this is true for any $\alpha\in\N$, $D_tu$ is $C^\infty$ in $U$, which concludes the proof of Lemma \ref{e:propagregu}.
\end{proof}

We conclude the proof of Theorem \ref{t:singprop}. We assume that 
\begin{equation}\label{e:smoothdebut}
\text{$u$ is smooth in $W=\mathscr{S}((-T-\delta_0,-T+\delta_0);(x_0,\xi_0))$.}
\end{equation}
Then, $u$ is smooth in a slightly larger set $W'$, i.e., such that $\overline{W}\subset W'$. By Lemma \ref{l:Sinnercontinuous}, there exists $V'\subset T^*\Omega\setminus 0$ an open neighborhood of $(x_0,\xi_0)$ such that 
$$
W\subset\mathscr{S}((-T-\delta_0,-T+\delta_0); V')\subset W'.
$$
Fix also an open set $V\subset T^*\Omega\setminus 0$ such that
$$
(x_0,\xi_0)\in V\subset \overline{V}\subset V'.
$$
Lemma \ref{e:propagregu} implies that $u$ is smooth in $\mathscr{S}((-T+\delta_0,0); V)$. In particular, 
\begin{equation}\label{e:smoothend}
\text{$u$ is smooth in $\mathscr{S}((-T+\delta_0,0); (x_0,\xi_0))$.}
\end{equation}

\smallskip

The fact that \eqref{e:smoothdebut} implies \eqref{e:smoothend} proves that singularities of \eqref{e:subwave} propagate only along \emph{rays}. Using that singularities of $P$ are contained in $\{p=0\}$, we obtain finally Theorem \ref{t:singprop}.

\section{Proof of Corollary \ref{c:singpropag}} \label{s:xyproof}
In all the sequel, that is, in Sections \ref{s:xyproof} and \ref{s:th1}, we assume that $A$ is a sub-Laplacian. As mentioned in Definition \ref{e:sL} (see also \cite{Hor67}), it means that we assume that $A$ has the form
\begin{equation} \label{e:sumofsquare}
A=\sum_{i=1}^K Y_i^*Y_i
\end{equation}
where the global smooth vector fields $Y_i$ satisfy Hörmander's condition. 

\smallskip

The principal symbol of $A$, which is also the natural Hamiltonian, is
\begin{equation*}
a=\sum_{i=1}^{K} h_{Y_{i}}^2.
\end{equation*}
Here, for $Y$ a vector field on $X$, we denoted by $h_{Y}$ the momentum map given in canonical coordinates $(x,\xi)$ by $h_{Y}(x,\xi)=\xi(Y(x))$. 

\smallskip

In this Section, we prove Corollary \ref{c:singpropag}. For that purpose, we introduce in Section \ref{s:sR} notations and concepts from sub-Riemannian geometry, which is a natural framework to study geometric properties of the vector fields $Y_1,\ldots,Y_K$. Our presentation is inspired by \cite[Chapter 5 and Appendix D]{Mon02} (see also \cite{ABB19}).

\subsection{Sub-Riemannian geometry and horizontal curves} \label{s:sR}
We consider the sub-Riemannian distribution
$$\mathcal{D}=\text{Span}(Y_{1},\ldots,Y_{K})\subset TX.
$$
There is a metric $g$, namely
\begin{equation} \label{e:metric}
g_x(v)=\inf\left\{\sum_{i=1}^{K} u_i^2 \ | \ v=\sum_{i=1}^{K} u_iY_{i}(x) \right\},
\end{equation}
which is Riemannian on $\mathcal{D}$ and equal to $+\infty$ outside $\mathcal{D}$. The triple $(X,\mathcal{D},g)$ is called a sub-Riemannian structure (see \cite{Mon02}). 

\smallskip

Fix an interval $I=[b,c]$ and a point $x_0\in X$. We denote by $\Omega(I,x_0;\mathcal{D})$ the space of all absolutely continuous curves $\gamma:I\rightarrow X$ that start at $\gamma(b)=x_0$ and whose derivative is square integrable with respect to $g$, implying that the length
\begin{equation*}
\int_I \sqrt{g_{\gamma(t)}(\dot{\gamma}(t))}dt
\end{equation*}
of $\gamma$ is finite. Such a curve $\gamma$ is called \emph{horizontal}. The \emph{endpoint map} is the map 
\begin{equation*}
\End:\Omega(I,x_0;\mathcal{D})\rightarrow X, \quad \gamma\mapsto\gamma(c).
\end{equation*}
The metric \eqref{e:metric} induces a distance $d$ on $X$, and $d(x,y)<+\infty$ for any $x,y\in X$ thanks to Hörmander's condition (this is the Chow-Rashevskii theorem).

\smallskip

Two types of curves in $\Omega(I,x_0;\mathcal{D})$ will be of particular interest: the critical points of the endpoint map, and the curves which are projections of the Hamiltonian vector field $H_a$ associated to $a$.

\smallskip

Projections of integral curves of $H_a$ are geodesics:
\begin{theorem} \cite[Theorem 1.14]{Mon02} 
Let $\gamma(s)$ be the projection on $X$ of an integral curve (in $T^*X$) of the Hamiltonian vector field $H_a$. Then $\gamma$ is a horizontal curve and every sufficiently short arc of $\gamma$ is a minimizing sub-Riemannian geodesic (i.e., a minimizing path between its endpoints in the metric space $(X,d)$).
\end{theorem}
\noindent Such horizontal curves $\gamma$ are called \textit{normal} geodesics, and they are smooth.

\noindent The differentiable structure on $\Omega(I,x_0;\mathcal{D})$ described in \cite[Chapter 5 and Appendix D]{Mon02} allows to give a sense to the following notion: 
\begin{definition}\label{d:singular}
A singular curve is a critical point for the endpoint map.
\end{definition}
\noindent Note that in Riemannian geometry (i.e., for $a$ elliptic), there exist no singular curves. 

\noindent In the next definition, we use the notation $\mathcal{D}^\perp$ for the annihilator of $\mathcal{D}$ (thus a subset of the cotangent bundle $T^*X$),  and $\overline{\omega}_X$ denotes the restriction to $\mathcal{D}^\perp$ of the canonical symplectic form $\omega_{X}$ on $T^*X$.
\begin{definition} \label{d:charac}
A characteristic for $\mathcal{D}^{\perp}$ is an absolutely continuous curve $\lambda(t)\in\mathcal{D}^\perp$ that never intersects the zero section of $\mathcal{D}^\perp$ and that satisfies $\dot{\lambda}(t)\in{\rm ker}(\overline{\omega}_X(\lambda(t)))$ at every point $t$ for which the derivative $\dot{\lambda}(t)$ exists.
\end{definition}
\begin{theorem}\cite[Theorem 5.3]{Mon02} \label{t:singprojchar}
A curve $\gamma\in\Omega$ is singular if and only if it is the projection of a characteristic $\lambda$ for $\mathcal{D}^\perp$ with square-integrable derivative. $\lambda$ is then called an abnormal extremal lift of the singular curve $\gamma$.
\end{theorem}
Normal geodesics and singular curves are particularly important in sub-Riemannian geometry because of the following fact (Pontryagin's maximum principle):
\begin{center}
 any minimizing geodesic in $(X,d)$ is either a singular curve or a normal geodesic.
\end{center}
The existence of minimizing geodesics which are singular curves but not normal geodesics was proved in \cite{Mon94}.

\smallskip

Let us mention three examples which are well-known in sub-Riemannian geometry (see \cite{Mon02} and \cite{ABB19}) where the singular curves and their abnormal extremal lifts can be explicitly computed. These examples are presented in $\R^n$ ($n=3,4$) for simplicity, but they could also have been written on adequate compact manifolds in order to fit better with the framework of the present work.
\begin{exa} \label{e:Heis}
When the vector fields are $Y_1=\partial_x-\frac{y}{2}\partial_z$ and $Y_2=\partial_y+\frac{x}{2}\partial_z$ in $\R^3$, and the measure is $\nu=dxdydz$, then 
$$A=-\left[\left(\partial_x-\frac{y}{2}\partial_z\right)^2+\left(\partial_y+\frac{x}{2}\partial_z\right)^2\right]$$ is the Heisenberg sub-Laplacian. In this case, the only singular curves are the trivial constant ones $t\mapsto q_0$ for any $q_0=(x_0,y_0,z_0)\in \R^3$, and the abnormal extremal lifts are given by $(q_0,p_0)\in T^*\R^3$ where $p_0\neq 0$ annihilates $Y_1(q_0)$ and $Y_2(q_0)$ and is thus proportional to $dz+\frac{y_0}{2}dx-\frac{x_0}{2}dy$.
\end{exa}
\begin{exa} \label{e:Martinet}
When the vector fields are $Y_1=\partial_x$ and $Y_2=\partial_y+x^2\partial_z$ in $\R^3$, and the measure is $\nu=dxdydz$, then 
$$A=-\left[\partial_x^2+\left(\partial_y+x^2\partial_z\right)^2\right]$$ is the Martinet sub-Laplacian. Apart from the trivial constant ones, the singular curves are of the form $t\mapsto q(t)=(0,t,z_0)$ for any $z_0\in \R$ and the abnormal extremal lifts are given by $(q(t),p(t))\in T^*\R^3$ where $p(t)\neq 0$ annihilates $Y_1(q(t))=\partial_x$ and $Y_2(q(t))=\partial_y$. Note that many other parametrizations of these curves are possible.
\end{exa}
\begin{exa}
When the vector fields are $Y_1=\partial_x-\frac{y}{2}\partial_z$, $Y_2=\partial_y+\frac{x}{2}\partial_z$ and $Y_3=\partial_w$ in $\R^4$, and the measure is $\nu=dxdydzdw$, then 
$$A=-\left[\left(\partial_x-\frac{y}{2}\partial_z\right)^2+\left(\partial_y+\frac{x}{2}\partial_z\right)^2+\partial_w^2\right]$$ is the quasi-contact (or even contact) sub-Laplacian.  Apart from the trivial constant ones, the singular curves are of the form $t\mapsto q(t)=(x_0,y_0,z_0,t)$ for any $x_0,y_0,z_0\in \R$, and their abnormal extremal lifts are given by $(q(t),p(t))\in T^*\R^4$ where $p(t)\neq 0$ annihilates $Y_1(q(t)), Y_2(q(t))$ and $Y_3(q(t))$, and is thus proportional to $dz+\frac{y_0}{2}dx-\frac{x_0}{2}dy$. Again, many other parametrizations of these curves are possible.
\end{exa}

\subsection{End of the proof of Corollary \ref{c:singpropag}} \label{s:proofcorollary}

The basic facts of sub-Riemannian geometry recalled in Section \ref{s:sR} now allow us to deduce Corollary \ref{c:singpropag} from Theorem \ref{t:singprop}.

\smallskip

We recall a few notations already used: $\pi$ denotes the canonical projection $\pi:T^*X\rightarrow X$, and $\mathcal{I}$ is the canonical isomorphism between $T(T^*X)$ and $T^*(T^*X)$ introduced in \eqref{e:Iisom}. The notation $a_{m}$ stands for half the Hessian of the principal symbol of $A$ at $m$, and $a_m^*$ is then defined as in \eqref{e:S*}.

\smallskip

The expression \eqref{e:gammam2} of the cones $\Gamma_m$ can be simplified thanks to the following lemma.
\begin{lemma} \label{l:eqamg}
When $A$ is a sub-Laplacian, there holds $a_{m}^{*}(\mathcal{I}(b ))= g(d\pi(b ))$ for any $b \in({\rm ker}(a_{m}))^{\perp_{\omega_X}}\subset T(T^*X)$.
\end{lemma}
\begin{proof} 
We consider a frame $Z_1,\ldots,Z_N$ which is $g$-orthonormal at the (canonical) projection of $m$ on $X$. In particular, the $Z_j$ are independent, and the $H_{h_{Z_j}}$ are also independent.
We have $a_{m}=\sum_{j=1}^N (dh_{Z_j})^2$. 
Hence, $H_{h_{Z_1}},\ldots,H_{h_{Z_N}}$ span $({\rm ker}(a_{m}))^{\perp_{\omega_X}}$ since
\begin{align*}
{\rm ker}(a_{m})&=\bigcap_{j=1}^N {\rm ker}(dh_{Z_j})=\{\xi\in T(T^*X), \ dh_{Z_j}(\xi)=0, \ \forall 1\leq j\leq N\}\\
&=\{\xi\in T(T^*X),\ \omega(\xi, H_{h_{Y_N}})=0, \ \forall 1\leq j\leq N\}\\
&=\text{span}(H_{h_{Y_1}},\ldots,H_{h_{Y_N}})^{\perp_{\omega_X}}.
\end{align*}
We fix $b \in({\rm ker}(a_{m}))^{\perp_{\omega_X}}$ and we write $b =\sum_{j=1}^N u_jH_{h_{Z_{j}}}$. By definition, $\mathcal{I}(H_{h_{Z_j}})=-dh_{Z_j}$ and $d\pi(H_{h_{Z_j}})=Z_j$ for any $j$, so there holds
\begin{align*}
&a_{m}^*\left(\mathcal{I}\left(\sum_{j=1}^N u_j H_{h_{Z_j}}\right)\right)=a_{m}^*\left(\sum_{j=1}^N u_j dh_{Z_{j}}\right)=\sup_{\eta\notin{\rm ker}(a_{m})} \frac{\left(\sum_{j=1}^N u_jdh_{Z_j}(\eta)\right)^2}{\sum_{j=1}^N dh_{Z_{j}}(\eta)^2}\\
&\qquad=\sup_{(\theta_j)\in \R^N}\frac{\left(\sum_{j=1}^N u_j\theta_j\right)^2}{\sum_{j=1}^N \theta_j^2}=\sum_{j=1}^N u_j^2=g\left(\sum_{j=1}^N u_jZ_j\right)\\
&\qquad =g\left(d\pi\left(\sum_{j=1}^N u_jH_{h_{Z_j}}\right)\right)
\end{align*}
where, to go from line 1 to line 2, we used that the $dh_{Z_j}$ are independent.
\end{proof}

From Lemma \ref{l:eqamg}, we deduce the following expression for $\Gamma_m$:
\begin{equation} \label{e:gammam3}
\begin{split}
& \qquad \qquad \Gamma_m=\R^+(\partial_t+B ), \\
& B=\left\{b \in {\rm ker}(a_m)^{\perp_{\omega_{X}}}, \ g(d\pi(b))\leq 1\right\}.
\end{split}
\end{equation}

Recall that $M=T^*(\R\times X)\setminus 0$, and let $\pi_2:M\rightarrow T^*X$ be the canonical projection on the second factor.
\begin{proposition} \label{p:diffrays}
Let $\gamma:I\rightarrow M$ be a null-ray,  as introduced in Definition \ref{d:defray}. Then, $\tau$ is constant along this null-ray, and necessarily:
 \begin{enumerate}[(i)]
 \item if $\tau\equiv c\neq 0$, then $\gamma$ is a null-bicharacteristic;
 \item if $\tau\equiv 0$, then $\gamma$ is contained in $\Sigma_{(2)}$ and tangent to the cones $\Gamma_m$ given by \eqref{e:gammam3}. Moreover, $\pi_2(\gamma)\subset T^*X$ is a characteristic curve, and its projection on $X$ (a singular curve) is traveled at speed $\leq 1$.
 \end{enumerate}
\end{proposition}  
\begin{proof}
First, we note that $\tau$ is constant along null-rays since $\tau$ is preserved along integral curves of $H_p$, and $d\tau(v)=0$ for any $v\in\Gamma_m$ when $\Gamma_m$ is given by \eqref{e:gammam3}. 

\smallskip

If $\tau\equiv c\neq 0$, then $\Gamma_m=\R^{\pm}\cdot H_p(m)$ for $m\in\gamma(I)$. Thus $\gamma$ is a null-bicharacteristic, which proves (i).

\smallskip

We finally prove (ii); we assume that $\tau\equiv 0$ along $\gamma$. Let $s\in I$. We set $n(s)=\pi_2(\gamma(s))$. According to \eqref{e:gammam3}, we can write $\dot{\gamma}(s)=c(s)(\partial_t+b(s))$ with $b(s) \in T_{n(s)}\mathcal{D}^\perp$ since $a\equiv 0$ along the path. There holds ${\rm ker}(a_{n(s)})=T_{n(s)}\mathcal{D}^\perp$ where $a_{n(s)}$ is half the Hessian of $a$ at point $n(s)$. Plugging into the above formula, we also get $b(s) \in (T_{n(s)}\mathcal{D}^\perp)^{\perp_{\omega_{X}}}$. It follows that 
$$b(s) \in T_{n(s)}\mathcal{D}^\perp\cap(T_{n(s)}\mathcal{D}^\perp)^{\perp_{\omega_{X}}}={\rm ker}(\overline{\omega}_X(n(s))).$$
This implies that $\pi_2(\gamma)$ is a characteristic curve. Its projection on $X$ is a singular curve, by definition. Moreover, the inequality $g(d\pi(b))\leq 1$ in \eqref{e:gammam3} exactly means that this projection is traveled at speed $\leq 1$.
\end{proof}

Proposition \ref{p:diffrays} directly implies Corollary \ref{c:singpropag}. To sum up, singularities of the wave equation \eqref{e:subwave} when $A$ is a sub-Laplacian propagate only along integral curves of $H_a$ and characteristics for $\mathcal{D}^\perp$ (at speed $\leq 1$).

\paragraph{Comments.} The propagation of singularities at speeds $< 1$ along singular curves is not excluded by Corollary \ref{c:singpropag}. If such a slow propagation effectively exists (which is not proved by Corollary \ref{c:singpropag}), it is in strong contrast with the usual propagation ``at speed $1$'' along the integral curves of $H_a$ (as in Hörmander's theorem). As already mentioned in the introduction, we proved this surprising fact in a joint work with Yves Colin de Verdi\`ere \cite{CdVL21}: we gave explicit examples of initial data of a subelliptic wave equation whose singularities effectively propagate at any speed between $0$ and $1$ along a singular curve.

\smallskip

 Note that a similar phenomenon occurs for ``partial sub-Laplacians'' (typically, $\partial_x^2$ in $\R^2_{x,y}$), as proved in \cite{MU79} using a parametrix construction. The result of \cite{MU79} is presented as a model for the ``conical refraction'', which is the splitting of a ray  into a cone of rays by a biaxial crystal.

\section{Proof of Theorem \ref{t:inclusion} and Corollary \ref{c:xy}} \label{s:th1}

We now turn to the study of the wave kernel $K_G$. Section \ref{s:th1} is devoted to the proof of Theorem \ref{t:inclusion} and Corollary \ref{c:xy}, i.e., we deduce the wave-front set of the Schwartz kernel $K_G$ from the ``geometric'' propagation of singularities given by Corollary \ref{c:singpropag}. The idea is to consider $K_G$ itself as the solution of a subelliptic wave equation to which we can apply Corollary \ref{c:singpropag}. 

\subsection{$K_G$ as the solution of a wave equation}

We consider the product manifold $X\times X$, with coordinate $x$ on its first copy, and coordinate $y$ on its second copy. We set
\begin{align*}
A^\otimes=\frac12(A_x\otimes {\rm Id}_y+{\rm Id}_x\otimes A_y)
\end{align*}
 and we consider the operator
$$
P^\otimes=\partial_{tt}^2-A^\otimes
$$
acting on functions of $\R\times X_x\times X_y$.  Using \eqref{e:Schwartzkernel}, we can check that the Schwartz kernel $K_G$ is a solution of 
$${K_G}_{|t=0}=0, \qquad \partial_t{K_G}_{|t=0}=\delta_{x-y}, \qquad P^\otimes K_G=0.$$

\smallskip

The operator $A^\otimes$ is a self-adjoint non-negative real second-order differential operator on $X\times X$. Moreover it is subelliptic: it is immediate that the vector fields $Y_1\otimes {\rm Id}_y, \ldots, Y_K\otimes {\rm Id}_y, {\rm Id}_x\otimes Y_1,\ldots, {\rm Id}_x\otimes Y_K$ verify Hörmander's Lie bracket condition, since it is satisfied by $Y_1,\ldots,Y_K$. Hence, Theorem \ref{t:singprop} applies to $P$, with the cones (and the null-rays) being computed with $A^\otimes$ in $T^*(X\times X)$ instead of $A$ (see \eqref{e:gammam2hat}). We denote by  $\sim_t$ the relation of existence of a null-ray of length $|t|$ joining two given points of $T^*(X\times X)\setminus 0$ (see Remark \ref{r:nullray} for the omission of the variables $t$ and $\tau$ in the null-rays). 

\smallskip

Since $WF(K_G(0))=\emptyset$ and
$$
WF(\partial_tK_G(0))=\{(z,z,\zeta,-\zeta)\in T^*(X\times X)\setminus 0\},
$$
(see \cite[p. 93]{RS75}) we have, according to Theorem \ref{t:singprop},
\begin{equation}\label{e:incl1}
\begin{split}
WF(K_G(t))\subset \{(x,y,\xi,-\eta)\in T^*(X\times X)\setminus 0, &\ \exists (z,\zeta)\in T^*X\setminus 0,  \\
& (z,z,\zeta,-\zeta)\sim_t (x,y,\xi,-\eta)\}.
\end{split}
\end{equation}
Our goal is now to give a simpler expression for the right-hand side of \eqref{e:incl1}.

\smallskip

Let us denote by $g^1$ the sub-Riemannian metric on $X_x$ and by $g^2$ the sub-Riemannian metric on $X_y$. The sub-Riemannian metric on $X_x\times X_y$ is $g^\otimes=\frac12(g^1\oplus g^2)$. In other words, if $q=(q_1,q_2) \in X\times X$ and $v=(v_1,v_2)\in T_q(X\times X)\approx T_{q_1}X\times T_{q_2}X$, we have
\begin{equation} \label{e:splitmetric}
g^\otimes_q(v)=\frac12(g^1_{q_1}(v_1)+g^2_{q_2}(v_2)).
\end{equation}
Now, according to \eqref{e:gammam3}, the cones $\Gamma^\otimes_m$ associated to $A^\otimes$ are given by
\begin{equation} \label{e:gammam2hat}
\begin{split}
& \qquad \qquad \Gamma^\otimes_m=\R^+(\partial_t+B ), \\
& B=\{b \in {\rm ker}(a^\otimes_m)^{\perp_{\omega^\otimes}}, \ g^\otimes(d\pi^\otimes(b))\leq 1\}.
\end{split}
\end{equation}
Here, $\perp_{\omega^\otimes}$ designates the symplectic orthogonal with respect to the canonical symplectic form $\omega^\otimes$ on $T^*(X\times X)$, and $\pi^\otimes:T^*(X\times X)\rightarrow X\times X$ is the canonical projection.

\smallskip
 
To evaluate the right-hand side of \eqref{e:incl1}, we denote by $\approx_t$ the relation of existence of a null-ray of length $|t|$ joining two given points of $T^*X\setminus 0$ (the cones $\Gamma_m$ are subsets of $T(T^*(\R\times X))$ as defined in Section \ref{s:thecones}). Let us prove that
\begin{equation} \label{e:inclWF}
\begin{split}
& \{(x,y,\xi,-\eta)\in T^*(X\times X)\setminus 0, \ \exists (z,\zeta)\in T^*X\setminus 0, \ (z,z,\zeta,-\zeta)\sim_t (x,y,\xi,-\eta)\}\\
& \qquad \qquad \subset \{(x,y,\xi,-\eta)\in T^*(X\times X)\setminus 0, \ (x,\xi)\approx_t (y,\eta)\}.
\end{split}
\end{equation}
 Combining with \eqref{e:incl1}, it will immediately follow that 
\begin{equation} \label{e:inclkG}
WF(K_G(t))\subset \{(x,y,\xi,-\eta)\in T^*(X\times X)\setminus 0, \ (x,\xi)\approx_t (y,\eta)\}.
\end{equation}

\subsection{Proof of \eqref{e:inclWF}.}
 We denote by $\gamma:[0,t]\rightarrow T^*(X\times X)\setminus 0$ a null-ray from $(z,z,\zeta,-\zeta)$ to $(x,y,\xi,-\eta)$, parametrized by time. Our goal is to construct a null-ray of length $|t|$ in $T^*X\setminus 0$, from $(y,\eta)$ to $(x,\xi)$. It is obtained by concatenating a null-ray from $(y,\eta)$ to $(z,\zeta)$ with another one, from $(z,\zeta)$ to $(x,\xi)$. However, there are some subtleties hidden in the parametrization of this concatenated null-ray.
 
 \smallskip
 
 We write $\gamma(s)=(\alpha_1(s),\alpha_2(s),\beta_1(s),\beta_2(s))$, and for $i=1,2$ and $0\leq s\leq t$, we set $\gamma_i(s)=(\alpha_i(s),\beta_i(s))\in T^*X$. We also set $\delta_i(s)=g^i(d\pi_i(\dot{\gamma}_i(s)))$, where $\pi_i:T^*X\rightarrow X$ (here $X$ is the $i$-th copy of $X$). The upper dot denotes here and in the sequel the derivative with respect to the time variable. Since $g^\otimes(d\pi^\otimes(\dot{\gamma}(s)))\leq 1$ for any $s\in [0,t]$, we deduce from \eqref{e:splitmetric} that
$$
\frac12(\delta_1(s)+\delta_2(s))\leq 1.
$$
Note that it is possible that $\delta^1(s)>1$ or $\delta^2(s)>1$.

\smallskip 

We are going to construct a null-ray $\varepsilon:[0,t]\rightarrow T^*X$ of the form 
\begin{align}
\e(s)&=(\alpha_2(\theta(s)), -\beta_2(\theta(s))), \qquad 0\leq s\leq s_0 \label{e:timereversal}\\
\e(s)&=(\alpha_1(\theta(s)),\beta_1(\theta(s))), \qquad s_0\leq s\leq t. \nonumber
\end{align}
The parameter $s_0$ and the parametrization $\theta$ will be chosen so that the first part of $\varepsilon$ joins $(y,\eta)$ to $(z,\zeta)$ and the second part joins $(z,\zeta)$ to $(x,\xi)$. We choose $\theta(0)=t$, hence $\varepsilon(0)=(y,\eta)$. Then, for $0\leq s\leq s_0$, we choose $\theta(s)\leq t$ in a way to guarantee that $g^1(d\pi_1(\dot{\e}(s)))= 1$. This defines $s_0$ in a unique way as the minimal time for which $\varepsilon(s_0)=(z,\zeta)$. In particular, $\theta(s_0)=0$. A priori, we do not know that $s_0\leq t$, but we will prove it below. Then, for $s\geq s_0$, we choose $\theta(s)\geq 0$ in order to guarantee that $g^2(d\pi_2(\dot{\e}(s)))= 1$. This defines a time $s_1$ in a unique way as the minimal time for which $\e(s_1)=(x,\xi)$. Finally, if $s_1\leq t$, we extend $\e$ by $\varepsilon(s)\equiv(x,\xi)$ for $s_1\leq s\leq t$.

\smallskip

We check that $\varepsilon$ is a null-ray in $T^*X$. We come back to the definition of null-rays as tangent to the cones $\Gamma_m$. It is clear that
$$
{\rm ker}(a^\otimes_m)^{\perp_{\omega^\otimes}}={\rm ker}(a_m)^{\perp_{\omega_1}}\times{\rm ker}(a_m)^{\perp_{\omega_2}}
$$
where $\omega_i$ is the canonical symplectic form on $T^*X_i$. Therefore, $\dot{\e}(s)\in {\rm ker}(a_m)^{\perp_{\omega_i}}$ for $i=1$ when $0\leq s\leq s_0$ and for $i=2$ when $s_0\leq s\leq t$. Thanks to Lemma \ref{l:eqamg}, the inequality in \eqref{e:gammam2} (but for the cones in $X_1$ and $X_2$) is verified by $\dot{\e}(s)$ for any $0\leq s\leq t$ by definition.
There is a ``time-reversion'' (or ``path reversion'') in the first line of \eqref{e:timereversal}; the property of being a null-ray is preserved under time reversion together with momentum reversion. Hence $\varepsilon$ is a null-ray in $T^*X$. 

\smallskip

The fact that $s_0,s_1\leq t$ follows from the following computation:
\begin{align*}
t\geq \int_0^t g^\otimes(d\pi^\otimes(\dot{\gamma}(s)))ds&=\frac12\int_0^{t}g^1(d\pi_1(\dot{\gamma}_1(s)))ds+\frac12\int_0^{t}g^2(d\pi_2(\dot{\gamma}_2(s)))ds\\
&=\frac12 \int_0^{s_0}g^1(d\pi_1(\dot{\e}(s)))ds+\frac12 \int_{s_0}^{s_1}g^2(d\pi_2(\dot{\e}(s)))ds \\
&=s_0+(s_1-s_0)=s_1.
\end{align*}
where the second equality follows from the fact that $\varepsilon$ is a reparametrization of $\gamma_1$ (resp. $\gamma_2$) for $s\in [0,s_0]$ (resp. $[s_0,s_1]$). This concludes the proof of \eqref{e:inclWF}.

\subsection{Conclusion of the proof of Theorem \ref{t:inclusion}}

%
Let us finish the proof of Theorem \ref{t:inclusion}. 
We fix $(x_0,\xi_0)$, $(y_0,\eta_0)$ and $t_0$ such that there is no null-ray from $(y_0,\eta_0)\in T^*X$ to $(x_0,\xi_0)\in T^*X$ in time $t_0$. 

\smallskip

\emph{Claim.} There exist a conic neighborhood $V$ of $(x_0,y_0,\xi_0,-\eta_0)$ in $T^*(X\times X)$ and a neighborhood $V_0$ of $t_0$ in $\R$ such that for any $N\in\N$ and any $t\in V_0$, $\partial_t^{2N}K_G(t)$ is smooth in  $V$.  

\begin{proof}
We choose $V$ so that for $(x,y,\xi,-\eta)\in V$ and $t\in V_0$, there is no null-ray from $(y,\eta)$ to $(x,\xi)$ in time $t$. Such a $V$ exists, since otherwise by extraction of null-rays (which are Lipschitz with a locally uniform constant, see \eqref{e:tislarge}), there would exist a null-ray from $(y_0,\eta_0)$ to $(x_0,\xi_0)$ in time $t_0$. Then, we can check that for any $N\in\N$, $K_G^{(2N)}=\partial_{t}^{2N}K_G$ is a solution of
$${K_G^{(2N)}}_{|t=0}=0, \qquad \partial_t{K_G^{(2N)}}_{|t=0}=(A^\otimes)^N\delta_{x-y}, \qquad P^\otimes K_G^{(2N)}=0.$$
Repeating the above argument leading to \eqref{e:inclkG} with $K_G^{(2N)}$ instead of $K_G$, we obtain
$$
WF(K_G^{(2N)}(t))\subset \{(x,y,\xi,-\eta)\in T^*(X\times X)\setminus 0, \ (x,\xi)\approx_t (y,\eta)\},
$$
which proves the claim.
\end{proof}

We deduce from the claim that if there is no null-ray from $(y_0,\eta_0)\in T^*X$ to $(x_0,\xi_0)\in T^*X$ in time $t_0$, then $(t_0,\tau_0,x_0,y_0,\xi_0,-\eta_0)\notin WF(K_G)$ for any $\tau_0\in \R$.

\smallskip

Finally, if there is a null-ray from $(y_0,\eta_0)$ to $(x_0,\xi_0)$ in time $t_0$, then $a(x_0,\xi_0)=a(y_0,\eta_0)$, and due to the fact that $WF(K_G)$ is included in the characteristic set of $\partial_{tt}^2-A^\otimes$, the only $\tau_0$'s for which $(t_0,\tau_0,x_0,y_0,\xi_0,-\eta_0)\in WF(K_G)$ is possible are the ones satisfying $\tau_0^2=a(x_0,\xi_0)=a(y_0,\eta_0)$. This concludes the proof of Theorem \ref{t:inclusion}.

\begin{remark} \label{r:biblio} Theorem \ref{t:inclusion} allows to recover some results already known in the literature.

In the situations studied in \cite{Las82}, \cite{LasLas82} and \cite{Mel86}, $\Sigma_{(2)}$ is a symplectic manifold (a typical example is given by Example \ref{e:Heis}). In this case, thanks to \eqref{e:gammam2}, we see that the only null-rays starting from points in $\Sigma_{(2)}$ are lines in $t$. Therefore Theorem \ref{t:inclusion} implies:
\begin{itemize}
\item the ``wave-front part'' of the main results of \cite{Las82} and \cite{LasLas82} (but not the effective construction of parametrices handled in these papers);
\item Theorem 1.8 in \cite{Mel84}.
\end{itemize}
\end{remark}

\subsection{Proof of Corollary \ref{c:xy}} \label{s:singsuppG}

We finally prove Corollary \ref{c:xy}. We fix $x,y\in X$ with $x\neq y$ and we denote by $T_{s}$ the minimal length of a singular curve joining $x$ to $y$. 

\smallskip

 We consider $\varphi:\R\rightarrow \R\times X\times X$, $t\mapsto (t,x,y)$ which has conormal set $N_{\varphi}=\{(t,x,y,0,\xi,\eta)\}$ (in other words $N_\varphi$ corresponds to $\tau=0$). Using Theorem \ref{t:inclusion} and Proposition \ref{p:diffrays}, we see that $WF(\mathscr{G})$ does not intersect the conormal set of $\varphi_{|(-T_{s},T_{s})}$. Then, \cite[Theorem 2.5.11']{FIOs} ensures that $\mathscr{G}$, which is the pull-back of $K_G$ by $\varphi_{|(-T_{s},T_{s})}$, is well-defined as a distribution over $(-T_s,T_s)$. Of course, ${\rm Sing Supp}(\mathscr{G})$ is the projection of $WF(\mathscr{G})$ (for $|t|< T_{s}$).

\smallskip

By definition of $T_s$, for $|t|< T_{s}$, null-rays between $x$ and $y$ are contained in $\{\tau\neq 0\}$, thus they are null-bicharacteristics (see Proposition \ref{p:diffrays}). Hence, the singularities of $\mathscr{G}$ occur at times belonging to the set $\mathscr{L}$ of lengths of normal geodesics (for $\tau>0$, we obtain normal geodesics from $y$ to $x$, and for $\tau<0$, normal geodesics from $x$ to $y$).

\begin{remark}
If $x=y$, the same reasoning as in the proof of Corollary \ref{c:xy} says nothing more than ${\rm Sing Supp}(K_G(\cdot,x,x))\subset \R$ since for any point $(x,\xi)\in\mathcal{D}^\perp$ and any $t\in\R$, the constant path joining $(x,\xi)$ to $(x,\xi)$ in time $t$ is a null-ray (with $\tau\equiv0$). 
\end{remark}

\appendix 
\section{Appendix} 
\subsection{Sign conventions in symplectic geometry} \label{a:symp} \label{a:techn}

In the present work, we take the following conventions (the same as \cite{Hor07}, see Chapter 21.1): on a symplectic manifold with canonical coordinates $(x,\xi)$, the symplectic form is $\omega=d\xi\wedge dx$, and the Hamiltonian vector field $H_f$ of a smooth function $f$ is defined by the relation $\omega(H_f,\cdot)=-df(\cdot)$. In coordinates, it reads
\begin{equation*}
H_f=\sum_j (\partial_{\xi_j}f) \partial_{x_j} -(\partial_{x_j} f)\partial_{\xi_j}.
\end{equation*}
In these coordinates, the Poisson bracket is 
\begin{equation*}
\{f,g\}=\omega(H_f,H_g)=\sum_j (\partial_{\xi_j}f) (\partial_{x_j}g) -(\partial_{x_j} f)(\partial_{\xi_j}g),
\end{equation*}
which is also equal to $H_fg$ and $-H_gf$.

\subsection{Pseudodifferential operators} \label{a:pseudo}

This appendix is a short reminder on basic properties of pseudodifferential operators. Most proofs can be found in \cite{Hor07}. In this paper, we work with the class of polyhomogeneous symbols (defined below), which is slightly smaller than the usual class of symbols but has the advantage that the subprincipal symbol can be read easily when using the Weyl quantization (see \cite{Hor07}, the paragraph before Section 18.6).

\smallskip

 
We consider $\Omega$ an open set of a $d$-dimensional manifold, and $\mu$ a smooth volume on $\Omega$. The variable in $\Omega$ is denoted by $q$. Let $\pi:T^*\Omega\rightarrow \Omega$ be the canonical projection. 

\smallskip

$S_\hom^n(T^*\Omega)$ stands for the set of homogeneous symbols of degree $n$ with compact support in $\Omega$. We also denote by $S_{\rm phg}^n(T^*\Omega)$ the set of polyhomogeneous symbols of degree $n$ with compact support in $\Omega$. Hence, $a\in S_{\rm phg}^n(T^*\Omega)$ if $a\in C^\infty(T^*\Omega)$, the projection $\pi(\supp(a))$ is a compact of $\Omega$, and there exist $a_j\in S^{n-j}_{\text{hom}}(T^*\Omega)$ such that for any $N\in \mathbb{N}$, $a-\sum_{j=0}^N a_j\in S_{\rm phg}^{n-N-1}(T^*\Omega)$. We denote by $\Psi^n_\phg(\Omega)$ the space of polyhomogeneous pseudodifferential operators of order $n$ on $\Omega$, with a compactly supported kernel in $\Omega\times\Omega$.

\smallskip

We use the Weyl quantization denoted by $\Op:S^n_\phg(T^*\Omega)\rightarrow \Psi^n_\phg(\Omega)$. It is obtained by using partitions of unity and the formula in local coordinates
\begin{equation*}
\Op(a)f(q)=\frac{1}{(2\pi)^d}\int_{\R^d_{q'}\times \R^d_{p}}e^{i\langle q-q',p\rangle}a\left(\frac{q+q'}{2},p\right)f(q')dq'dp.
\end{equation*}
If $a$ is real-valued, then $\Op(a)^*=\Op(a)$. Moreover, with this quantization, the principal and subprincipal symbols of $A=\Op(a)$ with $a\sim \sum_{j\leq n} a_j$ are simply $\sigma_p(A)=a_n$ and $\sigma_{\text{sub}}(A)=a_{n-1}$ (usually, the subprincipal symbol is defined for operators acting on half-densities, but we make here the identification $f\leftrightarrow fd\nu^{1/2}$).

\smallskip

We also have the following properties:
\begin{enumerate}
\item If $A\in \Psi^l_\phg(\Omega)$ and $B\in \Psi^n_\phg(\Omega)$, then $[A,B]\in\Psi^{l+n-1}_\phg(\Omega)$. Moreover, $\sigma_p([A,B])=\frac{1}{i}\{\sigma_p(a),\sigma_p(b)\}$ where the Poisson bracket is taken with respect to the canonical symplectic structure of $T^*\Omega$.
\item If $X$ is a vector field on $\Omega$ and $X^*$ is its formal adjoint in $L^2(\Omega,\mu)$, then $X^*X$ is a second order pseudodifferential operator, with $\sigma_p(X^*X)=h_X^2$ and $\sigma_{\text{sub}}(X^*X)=0$. Here, for $X$ a vector field, we denoted by $h_{X}$ the momentum map given in canonical coordinates $(x,\xi)$ by $h_{X}(x,\xi)=\xi(X(x))$. 
\item If $A\in \Psi_\phg^n(\Omega)$, then A maps continuously the space $H^s(\Omega)$ to the space $H^{s-n}(\Omega)$. 
\end{enumerate}
Finally, we define the essential support of $A$, denoted by ${\rm essupp}(A)$, as the complement in $T^*\Omega$ of the points $(q,p)$ which have a conic-neighborhood $W$ so that $A$ is of order $-\infty$ in $W$.

\subsection{The cones $\Gamma_m$ as generalized Hamiltonians} \label{s:what}
In this section, we interpret the set $B=B(m)$ which appears in the formula \eqref{e:gammam2}, namely
$$B(m)=\left\{b \in {\rm ker}(a_m)^{\perp_{\omega_{X}}}, \ a_m^*(\mathcal{I}(b))\leq 1\right\},$$
 as a generalized Hamiltonian, just adapting the notion of Clarke generalized gradient (see \cite[Chapter 1.2]{Cla90}) to the ``Hamiltonian'' framework.

\begin{definition}
Let $f$ be an almost everywhere differentiable function on $T^*X$ and let $\Omega_f$ be the set of points where it is not differentiable. Its generalized Clarke Hamiltonian $\mathcal{H}f(x)$ at $x\in\Omega_f$ is the set
$$
\mathcal{H}f(x)={\rm cxhl}\left\{\underset{j\rightarrow +\infty}{\lim} H_f(x_j), \ x_j\rightarrow x, \ x_j\notin \Omega_f\right\}\subset T_x(T^*X)
$$
where ${\rm cxhl}$ denotes the convex hull.
\end{definition}
 The main result of this section is the following:
\begin{proposition} \label{p:limitingcones}
For any $m\in \Sigma_{(2)}$, $B(m)=\mathcal{H}\sqrt{a}(m)$.
\end{proposition}
This proposition, beside giving an alternative proof of Lemma \ref{l:inner}, draws a link between our computations and the Pontryagin maximum principle in the Clarke formulation, which asserts that any sub-Riemannian geodesic (see Section \ref{s:sR}) is a solution of the differential inclusion
$$
\dot{\gamma}(s)\in \mathcal{H}\sqrt{a}(\gamma(s)).
$$
 The projection of a null-ray in $T^*X$ is also by Definition \ref{d:defray} a solution of this differential inclusion, and this ``explains'' why abnormal extremals appear naturally in Corollary \ref{c:singpropag}.

\smallskip

Before proving Proposition \ref{p:limitingcones}, we introduce the ``fundamental matrix'' $F$ (see \cite[Section 21.5]{Hor07}) defined as follows:
\begin{equation} \label{e:fundmatrix}
\forall Y,Z\in T_{m}(T^*X), \qquad \omega_{X}(Y,FZ)=a_m(Y,Z).
\end{equation}
Here $a_m(Y,Z)=\frac12 (\Hess a)(m)(Y,Z)$. Then, $\omega_{X}(FY,Z)=-\omega_{X}(Y,FZ)$. As already explained in Section \ref{s:formula}, there is here a slight abuse of notations since $T_m(T^*X)$ stands for $T_{\pi_2(m)}(T^*X)$ where $\pi_2:M\rightarrow T^*X$ is the canonical projection on the second factor.

\begin{lemma} \label{l:funda}
The fundamental matrix induces an isomorphism
 \begin{equation*}
F:T_{m}(T^*X)/{\rm ker}(a_m)\rightarrow {\rm ker}(a_m)^{\perp_{\omega_{X}}}
\end{equation*}
\end{lemma}
\begin{proof} 
$F$ clearly passes to the quotient by ${\rm ker}(a_m)$ by \eqref{e:fundmatrix}. Let $b \in {\rm ker}(a_m)^{\perp_{\omega_{X}}}$. We set $b_0=-\mathcal{I}(b)\in {\rm ker}(a_m)^\perp$. The bilinear form $a_m$ is continuous and coercive on $T_m(T^*X)/{\rm ker}(a_m)$, and $b_0$ is a linear form on this space, thus by Lax-Milgram's lemma we get the existence of $Z\in T_m(T^*X)/{\rm ker}(a_m)$ such that $b_0=a_m(Z,\cdot)$. Finally, we have $$-\mathcal{I}(b)=b_0=a_m(\cdot,Z)=\omega(\cdot,FZ)=-\mathcal{I}(FZ)$$ according to \eqref{e:Iisom}, which means that $b=FZ$.
\end{proof}

Now we derive a formula for $B(m)$ in terms of the fundamental matrix (see formula (2.6) in \cite{Mel86}):
\begin{lemma}\label{l:formulaGammafundamatrix}
There holds
\begin{equation} \label{e:Gammam4}
B(m)={\rm cxhl}\left\{\frac{FZ}{a_m(Z)^{\frac12}}, \ Z\in T_{m}(T^*X)/{\rm ker}(a_m)\right\}.
\end{equation}
\end{lemma}
\begin{proof}
We have to compare \eqref{e:gammam2} with \eqref{e:Gammam4}.

\smallskip

First, let $b \in {\rm ker}(a_m)^{\perp_{\omega_{X}}}$ with $a_m^{*}(\mathcal{I}(b ))\leq 1$. By the proof of Lemma \ref{l:funda}, there exists $Z\in T_{m}(T^*X)/{\rm ker}(a_m)$ such that $-\mathcal{I}(b)=b_0=a_m(Z,\cdot)$. Using that $a_m^{*}(b_0)\leq 1$, we obtain $a_m(Z)\leq 1$, hence $b_0=\lambda a_m(Z,\cdot)/a_m(Z)^{\frac12}$ where $|\lambda|\leq 1$. It follows that $b =-\mathcal{I}^{-1}(b_0)=\lambda FZ/a_m(Z)^{\frac12}$. This proves that the cones given by \eqref{e:gammam2} are included in those given by \eqref{e:Gammam4}.

\smallskip

For the converse, we first notice that $FZ/a_m(Z)^{\frac12}\in {\rm ker}(a_m)^{\perp_{\omega_{X}}}$, and thus it is also the case for any convex combination. Also, it follows from the definitions of $\mathcal{I}$, $F$, $a_m^*$ and the Cauchy-Schwarz inequality that 
$$\forall Z\in T_{m}(T^*X)/{\rm ker}(a_m), \quad a_m^*(\mathcal{I}(FZ)/a_m(Z)^{\frac12})\leq 1.$$
By convexity of $a_m^*$, we obtain that any convex combination $b$ of elements of the form $FZ/a_m(Z)^{\frac12}$ satisfies $a_m^*(\mathcal{I}(b))\leq 1$. This concludes the proof.
\end{proof}

\begin{proof}[Proof of Proposition \ref{p:limitingcones}]
As in Section \ref{s:innerwithformula}, we work in a chart near $m$. Following the computations of Lemma \ref{l:linalg}, we have for any sequence of points $(m_j)_{j\in\N}$ such that $m_j-m\notin {\rm ker}(a_m)$,
\begin{align*}
\frac12\omega_{X}(H_a(m_j),w)&=-\frac12da(m_j)(w)=-a_m(m_j-m,w)+o(m_j-m)\\
&=\omega_{X}(F(m_j-m),w)+o(m_j-m),
\end{align*}
which implies
\begin{equation} \label{e:equivF}
H_{\sqrt{a}}(m_j)=\frac12\frac{H_a(m_j)}{a(m_j)^{\frac12}}=\frac{F(m_j-m)}{a_m(m_j-m)^{\frac12}}+o(1).
\end{equation}
Choosing $m_j=m+\e_j Z$ with $\e_j\rightarrow 0$, we obtain
$$
H_{\sqrt{a}}(m_j)\underset{j\rightarrow+\infty}{\longrightarrow}\frac{FZ}{a_m(Z)^{\frac12}}
$$
which proves that $B(m)\subset \mathcal{H}\sqrt{a}(m)$ according to Lemma \ref{l:formulaGammafundamatrix}.

\smallskip

Conversely, since $F$ is a linear isomorphism (see Lemma \ref{l:funda}), it is not difficult to see that any limit of $\frac{F(m_j-m)}{a_m(m_j-m)^{\frac12}}$ is of the form $\frac{FZ}{a_m(Z)^{\frac12}}$. Using \eqref{e:equivF} and taking convex hulls, this proves that $\mathcal{H}\sqrt{a}(m)\subset B(m)$.
\end{proof}
%
%
%
%
%

\end{document}